\documentclass[11pt]{amsart}
\usepackage{geometry}                
\usepackage{graphicx}
\usepackage{latexsym}
\usepackage{amsmath,amssymb,amsthm}
\usepackage{young}
\usepackage[boxsize=normal]{ytableau}
\usepackage{epstopdf}
\usepackage{tikz}

\usepackage[pdftex,bookmarks=false,colorlinks=true,linkcolor=black,urlcolor=black,citecolor=black,plainpages=false]{hyperref}
\newcommand{\Lg}{{\mathfrak{g}}}

\newcommand{\ud}{{\underline{{\mathbf d}}}}

\newcommand{\Waff}{{W^{{\mathfrak a}}}}
\def\mapright#1{\smash{\mathop{\longrightarrow}\limits^{#1}}}
\newcommand{\bU}{{\mathbb U}}
\newcommand{\T}{{\mathcal T}}

\newcommand{\Bc}{{\mathcal B}}

\newcommand{\ui}{{\underline{\bf i}}}

\newcommand{\bA}{{{\mathbb A}}}
\newcommand{\A}{{{\mathbb V}}}
\newcommand{\bW}{{{\mathbb W}}}
\newcommand{\bC}{{{\mathbb C}}}
\newcommand{\bc}{{{\mathbb C}}}
\newcommand{\Z}{{{\mathbb Z}}}
\newcommand{\bz}{{{\mathbb Z}}}
\newcommand{\R}{{{\mathbb R}}}

\newcommand{\Gl}{{\text{GL}}}
\newcommand{\semi}{{\,\rule[.1pt]{.4pt}{5.3pt}\hskip-1.9pt\times}}
\newcommand{{\sminus}}{\hbox{\tiny{r}}}
\newcommand{{\splus}}{\hbox{\tiny{r\hskip-2pt+\hskip-2pt 4}}}
\newcommand{\la}{\langle}
\newcommand{\ra}{\rangle}

\newcommand{\Mor}{{{\hbox{\rm Mor}}\,}}
\newcommand{\Aut}{{{\hbox{\rm Aut}}\,}}

\newcommand{\Lie}{{\mathop{\hbox{\rm Lie}}\,}}

\newcommand{\Pc}{\ensuremath{\mathcal{P}}}
\newcommand{\Qc}{\ensuremath{\mathcal{Q}}}

\newcommand{\co}{\mathcal{O}}
\newcommand{\ck}{\mathcal{K}}
\newcommand{\cg}{\mathcal{G}}
\newcommand{\cl}{\mathcal{L}}

\newcommand{\cb}{\mathcal{B}}

\newcommand{\cp}{\mathcal{P}}
\newcommand{\cq}{\mathcal{Q}}

\newcommand{\ens}[1]{\ensuremath{\left\{#1\right\}}}
\newcommand{\af}{\ensuremath{\mathfrak{a}}}

\newcommand{\lam}{{\lambda}}

\newtheorem{lem}{Lemma}
\newtheorem{defn}{Definition}
\newtheorem{thm}{Theorem}
\newtheorem*{thmohne}{Theorem}
\newtheorem{coro}{Corollary}
\newtheorem*{corohne}{Corollary}
\newtheorem{rem}{Remark}
\newtheorem{exam}{Example}

\newtheorem{prop}{Proposition}
\newcount\cols
{\catcode`,=\active\catcode`|=\active
 \gdef\YYoung(#1){\hbox{$\vcenter
 {\mathcode`,="8000\mathcode`|="8000
  \def,{\global\advance\cols by 1 &}%
  \def|{\cr
        \multispan{\the\cols}\hrulefill\cr
        &\global\cols=2 }%
  \offinterlineskip\everycr{}\tabskip=0pt
  \dimen0=\ht\strutbox \advance\dimen0 by \dp\strutbox
  \halign
   {\vrule height \ht\strutbox depth \dp\strutbox##
    &&\hbox to \dimen0{\hss$##$\hss}\vrule\cr
    \noalign{\hrule}&\global\cols=2 #1\crcr
    \multispan{\the\cols}\hrulefill\cr%
   }
 }$}}
 \gdef\Skew(#1:#2){\hbox{$\vcenter
 {\mathcode`,="8000\mathcode`|="8000
  \dimen0=\ht\strutbox \advance\dimen0 by \dp\strutbox
  \def\boxbeg{\vbox
    \bgroup\hrule\kern-0.4pt\hbox to\dimen0\bgroup\strut\vrule\hss$}%
  \def\boxend{$\hss\egroup\hrule\egroup}%
  \def,{\boxend\boxbeg}%
  \def|##1:{\boxend\vrule\egroup\nointerlineskip\kern-0.4pt
    \moveright##1\dimen0\hbox\bgroup\boxbeg}%
  \def\\##1\\##2:{\boxend\vrule\egroup\nointerlineskip\kern-0.4pt
    \kern ##1\dimen0\moveright##2\dimen0\hbox\bgroup\boxbeg}%
  \moveright#1\dimen0\hbox\bgroup\boxbeg#2\boxend\vrule\egroup
 }$}}
}

\def\mapri#1{\smash{\mathop{\longrightarrow}\limits^{#1}}}

\title{Knuth relations, tableaux and MV-cycles}
\author{St\'ephane Gaussent, Peter Littelmann and An Hoa Nguyen}

\address{St\'ephane Gaussent: \newline
Universit\'e de Lyon, Institut Camille Jordan (UMR 5208) 
Universit\'e Jean Monnet,\newline
23, rue du Docteur Michelon
42023 Saint-Etienne Cedex 2, France
}
\email{stephane.gaussent@univ-st-etienne.fr}

\address{Peter Littelmann:\newline
Mathematisches Institut, Universit\"at zu K\"oln,\newline
Weyertal 86-90, D-50931 K\"oln,Germany
}
\email{littelma@math.uni-koeln.de}

\address{An Hoa Nguyen:\newline
Mathematisches Institut, Universit\"at zu K\"oln,\newline
Weyertal 86-90, D-50931 K\"oln,Germany
}
\email{ahnguyen@math.uni-koeln.de}

\date{03.10.2012}                                           
\dedicatory{Dedicated to C. S. Seshadri on the occasion of his 80th birthday}
\begin{document}
\maketitle

{\bf Abstract: } We give a geometric interpretation of the Knuth equivalence relations in terms of the affine Gra\ss mann variety. The Young tableaux are seen as sequences of coweights, called galleries. We show that to any gallery corresponds a Mirkovi\'c-Vilonen cycle and that two galleries are equivalent if, and only if, their associated MV cycles are equal. Words are naturally identified to some galleries. So, as a corollary, we obtain that two words are Knuth equivalent if, and only if, their associated MV cycles are equal.

\section{Introduction}
In the theory of finite dimensional representations of complex 
reductive algebraic groups, the group $GL_n(\bC)$ is singled out by
the fact that besides the usual language of weight lattices, roots 
and characters, there exists an additional important combinatorial 
tool: the plactic monoid of Knuth \cite{knuth}, Lascoux and Sch\"utzenberger \cite{LaS}, and
the Young tableaux. 

The plactic monoid is the monoid of all words in the alphabet $\bA=\{1,\ldots,n\}$ modulo Knuth equivalence 
(see \eqref{knuthrelation}). The map associating to a word its class in the plactic monoid has a remarkable
section: the semi-standard Young tableaux. In the framework of crystal bases 
\cite{joseph}, \cite{hongkang}, \cite{kashiwaraone}, \cite{kashiwaratwo}, \cite{lusztig} and
the path model of representations \cite{Li2}, this monoid got a new interpretation, 
which made it possible to generalize the classical tableaux character formula to integrable highest 
weight representations of Kac-Moody algebras.

Mirkovi\'c and Vilonen \cite{MV} gave a geometric interpretation of weight multiplicities for finite
dimensional representations of a semi simple algebraic group $G$. Given a dominant coweight $\lambda^\vee$,
let $X_{\lambda^\vee}\subset \cg$ be the Schubert variety in the associated affine Gra\ss mann variety $\cg$
\cite{lusztig, Ku}. Let $U^-$ be the maximal unipotent subgroup opposite to a fixed Borel subgroup
and denote by $\co$ the ring of formal power series $\bc[[t]]$ and by $\ck$ its quotient field.
The irreducible components of $\overline{G(\co).\lam^\vee \cap U^-(\ck).\mu^\vee}\subset X_{\lambda^\vee}$ (Section~\ref{mvgallery})
are called MV-cycles and the number of irreducible components is the weight multiplicity of the weight $\mu^\vee$
in the complex irreducible representation $V(\lam^\vee)$ for the Langlands dual group $G^\vee$. 

This leads naturally to the question: is it possible to give a geometric formulation of the Knuth relations?
The aim of this article is to do this for $G=SL_n(\bc)$.  
For $1\le d\le n$ denote by $I_{d,n}$ the set $\{\ui=(i_1,\ldots,i_d)\mid 1\le i_1<i_2<\ldots i_d\le n\}$,
and set $\bW=I_{1,n}\cup I_{2,n}\cup\ldots\cup I_{n-1,n}\cup I_{n,n}$.
We can identify $\bA$ with $I_{1,n}$ and hence view $\bA$ as a subset of $\bW$. Words
in the alphabet $\bW$ are called {\it galleries}, and we introduce on the set of galleries an 
equivalence relation $\sim_K$ which, restricted to words in the alphabet $\bA$, is exactly the Knuth equivalence. 
To a given gallery $\gamma$, we associate 
in a canonical way a Bott-Samelson type variety $\Sigma$, a cell $C_\gamma\subset \Sigma$,
a dominant coweight $\nu^\vee$,
a map $\pi:\Sigma\rightarrow \cg$ and show (see Theorem~\ref{mainthm1}):
\begin{thmohne}
{\it a)} The closure of the image $\overline{\pi(C_\gamma)}$ is a MV-cycle in the Schubert variety $X_{\nu^\vee}$.

{\it b)} Given a second gallery $\gamma'$ with associated
Bott-Samelson type variety $\Sigma'$, cell $C_{\gamma'}\subset \Sigma'$ and map $\pi':\Sigma'\rightarrow \cg$,
then $\gamma\sim_K\gamma'$ if and only if $\overline{\pi(C_\gamma)}=\overline{\pi'(C_{\gamma'})}$.
\end{thmohne}
In particular we get:
\begin{corohne}
Two words in the alphabet $\bA$ are Knuth equivalent
if and only if the closure of the images of the corresponding cells define the same MV-cycle. 
\end{corohne}
The theorem  generalizes for $G=SL_n(\bc)$ the result in \cite{GL2}, where it was proved that
(for semisimple algebraic groups) the image of a cell associated to a Lakshmibai-Seshadri galleries is an MV-cycle.

The construction of the Bott-Samelson type variety and the cell can be sketched out as follows: the elements
of the alphabet can be identified with weights occurring in the fundamental representations of the group $SL_n(\bc)$
(for example $i\leftrightarrow \epsilon_i$), so words and galleries can be viewed as polylines in the apartment 
associated to the maximal torus $T$ of diagonal matrices in $G=SL_n(\bc)$.
The vertices and edges of a polyline $\gamma$ give naturally rise to a sequence $P_0\supset Q_0\subset P_1\supset\ldots\subset P_r$
of parabolic subgroups of the associated affine Kac-Moody group, and hence we can associate to a gallery
the Bott-Samelson type variety $\Sigma=P_0\times_{Q_0}  P_1\times_{Q_1} \ldots \times_{Q_{r-1}}P_r/P_r$. In the language of buildings 
(which we do not use in this article) the variety $\Sigma$ can be seen as the variety of all galleries in the building 
of the same type as $\gamma$ and starting in $0$, making $\Sigma$ into a canonical object associated to $\gamma$.
By choosing a generic anti-dominant one parameter subgroup $\eta$ of $T$, we get a natural 
Bia{\l}ynicki-Birula cell \cite{BB} decomposition of $\Sigma$. The $\eta$-fixed points in $\Sigma$ correspond exactly
to galleries of the same type as $\gamma$ (section~\ref{galandkey}), and 
$C_\gamma=\{x\in\Sigma\mid \lim_{t\rightarrow 0} \eta(t)x=\gamma\}$. The variety $\Sigma$ is
a desingularization of a Schubert variety in $\cg$ and is hence naturally endowed with a morphism
$\pi:\Sigma\rightarrow \cg$ \cite{GL2}.

Given a gallery $\gamma$, we get a natural sequence of vertices $(\mu^\vee_0,\ldots,\mu^\vee_r)$ crossed by the
gallery. We attach to each vertex a subgroup $U^-_j$ of $U^-(\ck)$ and show that the image of the cell
is $U^-_0\cdots U^-_r\mu^\vee$, where $\mu^\vee$ is the endpoint of the associated polyline.
Using this description of the image and commuting rules for root subgroups, we show that the image
of cells associated to Knuth equivalent galleries have a common dense subset. Since every equivalence
class has a unique gallery associated to a semi-standard Young tableaux, the results in \cite{GL2} show that the closure
of such an image is a MV-cycle, and this correspondence semi-standard Young tableaux $\leftrightarrow$ MV-cycles
is bijective.

We conjecture that, as in the special case of Lakshmibai-Seshadri galleries in \cite{GL2}, a generalization of the theorem above
holds (after an appropriate reformulation and using \cite{Li1} with Lakshmibai-Seshadri tableaux / galleries as section instead 
of Young tableaux) for arbitrary complex semi-simple algebraic groups.

\section{Words and keys}
\label{seWandK}
We fix as alphabet the set $\bA=\{1,2,\ldots,n\}$ with the usual order $1<2<\ldots<n$. By a word we mean a finite
sequence $(i_1,\ldots,i_k)$ of elements in $\bA$, $k$ is called the length of the word. The set of words ${\mathcal W}$ forms a monoid
with the concatenation of words as an operation and with the empty word as a neutral element. We define an equivalence relation
on the set of words as follows. We say that two words $u,w$ are related by a {\it Knuth relation} if we can find a decomposition of $u$ and $w$ 
such that $v_1,v_2$ are words, $x,y,z\in\bA$ and
\begin{equation}\label{knuthrelation}
\left\{    
\begin{array}{rl}
& u=v_1      xzy v_2\text{\ and\ }    w=v_1 zxy v_2\text{\ where\ }x<y\le z\\
or & u=v_1  yzx v_2\text{\ and\ }  w=v_1     yxz v_2\text{\ where\ }x\le y < z \\
\end{array}
\right. .
\end{equation}
\begin{defn}\label{knuthequivalence}\rm
Two words are called {\it Knuth equivalent}: $u\sim_K w$, if there exists a sequence of words $u=v_0,v_1,\ldots,v_r=w$
such that $v_{j-1}$ is related to $v_{j}$ by a {Knuth relation}, $j=1,\ldots,r$.
\end{defn}
\begin{rem}\rm
The set of equivalence classes forms a monoid (called the {\it plactic monoid}, see \cite{LaS,Sch,Li1})
with the concatenation of classes of words as an operation and with the class of the empty word as a neutral element.
\end{rem}
\begin{rem}\rm
Though we mostly refer to \cite{Ful,La,LaS} for results concerning the word algebra and Knuth relations,
note that there is a difference in the way the Knuth relations are stated.
The reason for the change is the way we translate tableaux into words, see Remark~\ref{remtabword}.
\end{rem}
\begin{defn}\rm
A {\it key diagram} of {\it column shape} $(n_1,\ldots,n_r)$ is a sequence of columns of boxes aligned in the top row, having $n_1$
boxes in the right most column, $n_2$ boxes in the second right most column etc. The key diagram is called a Young diagram
if $n_1\le \ldots\le n_r$.
\end{defn}
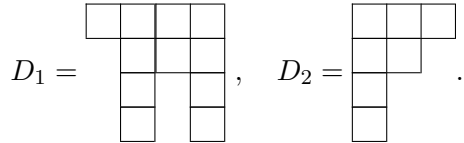
\begin{exam}\rm
$$
D_1=\Skew(0:\ ,\ |1:\ |1:\ |1:\ )\Skew(0:\ ,\ |0:\ ,\ |1:\ |1:\ )\ ,\quad D_2=\Skew(0: \ ,\ ,\ |0:\ ,\ |0:\ |0:\ ).
$$
The key diagram $D_1$ is of column shape $(4,2,4,1)$, the Young diagram $D_2$ is of column shape $(1,2,4)$
\end{exam}
\begin{defn}\rm
A {\it key tableau} of column shape $(n_1,\ldots,n_r)$ is a filling of the corresponding key diagram with elements from the alphabet $\bA$
such that the entries are {\it strictly increasing in the columns} (from top to bottom). A {\it semistandard Young tableau} of shape
$(n_1,\ldots,n_r)$, $n_1\le \ldots\le n_r$, is a filling of the corresponding Young diagram with elements from the alphabet $\bA$
such that the entries are strictly increasing in the columns (from top to bottom) and weakly increasing in the rows (left to right).
\end{defn}
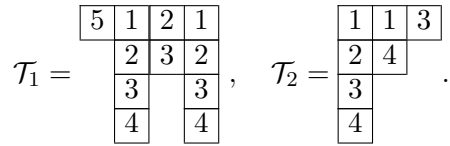
\begin{exam}\label{exam2}\rm
$$
\T_1=\Skew(0:5 ,1 |1:2 |1:3 |1:4)\Skew(0:2 ,1 |0:3 ,2 |1:3 |1:4)\ ,\quad \T_2=\Skew(0: 1 ,1 ,3 |0:2 ,4 |0:3 |0:4 ).
$$
The key  tableau  $\T_1$ is of column shape $(4,2,4,1)$, $\T_2$ is an example for a semistandard Young tableau
of column shape $(1,2,4)$.
\end{exam}
\begin{defn}\rm
Let $\T$ be a  key  tableau. The associated word $w_\T$ is the sequence $(i_1,\ldots,i_N)$ of elements in $\bA$
obtained from $\T$ by reading the entries in the key tableau boxwise, from top to bottom in each column, starting with the right most column.
\end{defn}
\begin{exam}\rm
For the key tableaux $\T_1$ and $\T_2$ in Example~\ref{exam2} we get 
$$
w_{\T_1}=(1,2,3,4,2,3,1,2,3,4,5),\quad w_{\T_2}=(3,1,4,1,2,3,4).
$$
The map $\T\mapsto w_\T$, which associates to a key tableaux the word in the alphabet $\bA$ is not bijective, in general there are
several key tableaux giving rise to the same word. All the
following key tableaux have $(1,2,3)$ as the associated word:
\begin{equation}\label{wordtableau}
\Skew(0:3,2,1)\,,\ \Skew(0:2,1|0:3)\,,\ \Skew(0:3,1|1:2)\,,\ \Skew(0:1|0:2|0:3).
\end{equation}
\end{exam}
\begin{rem}\label{remtabword}\rm
Let $\tilde w_\T$ be the word obtained from $w_\T$ by reading the word backwards, then 
$\tilde w_\T$ is the column word (the product of the column words, read from left to right, in each column from bottom to top),
which is Knuth equivalent (using the definition of the relation in \cite{LaS}) to the row word
as defined in \cite{LaS}. The change in the definition of the Knuth relation is due to the fact that
the "backwards reading map" $w\mapsto \tilde w$ respects the Knuth relations for $x<y<z$ but
not for $x=y<z$ respectively $x<y=z$ (see \cite{LaS}), but it sends equivalence classes for the relations 
in \eqref{knuthrelation} onto the equivalence classes for the Knuth relations as defined in \cite{LaS}.
\end{rem}

\begin{defn}\rm
Two key tableaux $\T,\T'$ are called {\it Knuth equivalent}: $\T\sim_K \T'$ if and only if the two associated words $w_\T$ and $w_{\T'}$ 
are Knuth equivalent, i.e.  $w_\T\sim_K w_{\T'}$.
\end{defn}
\begin{exam}\rm
The tableaux in \eqref{wordtableau} are all Knuth equivalent because they have the same associated word, but also
$$
\T_1=\Skew(0:2,1|1:3)\sim_K\T_2= \Skew(0:1,3|0:2)
$$
because $w_{\T_1}=(1,3,2)\sim_K w_{\T_2}=(3,1,2)$.
\end{exam}
Recall (see for example \cite{Ful} or \cite{Li1}):
\begin{thm}\label{thm1}
\begin{itemize}
\item[{\it i)}] Given a key tableau $\T$, there exists a unique semistandard Young tableau $\T'$ such that $\T\sim_K \T'$.
\item[{\it ii)}] Given a word $w$ in the alphabet $\bA$, there exists a unique semistandard Young tableau $\T$ such that $w\sim_K w_\T$.
\end{itemize}
\end{thm}
\begin{rem}\label{rem2}\rm
The unique semistandard Young tableau can be obtained from the given key tableau or  word by using the Jeu de Taquin algorithm 
of Lascoux and Sch\"utzenberger or the bumping algorithm, see for example \cite{Ful}. 
\end{rem}
\section{Galleries and keys}\label{galandkey}
For $1\le d\le n$ denote by $I_{d,n}$ the set of all strictly increasing sequences of length $d$:
$$
I_{d,n}=\{\ui=(i_1,\ldots,i_d)\mid 1\le i_1<i_2<\ldots i_d\le n\}.
$$ 
We fix a second alphabet 
$$
\bW=I_{1,n}\cup I_{2,n}\cup\ldots\cup I_{n-1,n}\cup I_{n,n}.
$$
We can identify $\bA$ with $I_{1,n}$ and hence view $\bA$ as a subset of $\bW$.
\begin{defn}\rm
A {\it gallery} is a finite sequence $\gamma= (\ui_1,\ldots,\ui_r)$ of elements in the alphabet $\bW$, $r$ is called the length of the gallery.
We say that the gallery is of type $(d_1,\ldots,d_r)$ if $\ui_1\in I_{d_1,n}$, $\ui_2\in I_{d_2,n}$, $\ldots$, $\ui_r\in I_{d_r,n}$.
\end{defn}
By the inclusion of $\bA\hookrightarrow \bW$ we can
identify words in the alphabet $\bA$ with galleries of type $(1,\ldots,1)$ (where the number of $1$'s is equal to the length of the word).
\begin{defn}\rm
Given a key tableau $\T$ of column shape $(n_1,\ldots,n_r)$, we associate to $\T$ the {\it gallery $\gamma_\T=(\ui_1,\ldots,\ui_t)$},
of type  $(n_1,\ldots,n_r)$, where the sequence $\ui_\ell=(j_1,\ldots,j_d)\in I_{d_\ell,n}$, $\ell=1,\ldots,r$, is obtained by reading the entries in the $\ell$-th column 
(counted from right to left) of the key tableau $\T$ from top to bottom.
\end{defn}
\begin{exam}\rm
For the key tableaux $\T_1$ and $\T_2$ in Example~\ref{exam2} we get the galleries
$$
\gamma_{\T_1}=((1,2,3,4),(2,3),(1,2,3,4),(5)) \ ,\quad\ \gamma_{\T_2}=((3),(1,4),(1,2,3,4)).
$$
\end{exam}
\begin{defn}\rm
Given a gallery $\gamma=(\ui_1,\ldots,\ui_r)$ of type $\ud=(d_1,\ldots,d_r)$, let $\T_\gamma$ be the key tableau of column shape $(d_1,\ldots,d_r)$
such that the entries in the $j$-th column (counted from right to left) are (counted from top to bottom) given by the sequence $\ui_j$.
\end{defn}
\begin{exam}\rm
For the gallery $\gamma=((3,6,7),(2),(2,3,4),(1,4))$ we get the key tableau:
$$
\T_\gamma=\Skew(0:1 ,2 |0:4 ,3 |1:4 )\Skew(0: 2,3|1:6|1:7).
$$
\end{exam}
One obtains a natural bijection between galleries and key tableaux, so in the following we identify often a gallery
with its key tableau and vice versa. For example, the word $w_\gamma$ associated to a gallery is the word associated
to the key tableau $\T_\gamma$, and we say that {\it two galleries are Knuth equivalent} if their corresponding key tableaux are Knuth equivalent.
\noindent
\section{The loop group $G(\ck)$ and the Kac-Moody group $\hat\cl(G)$}\label{loopgroup} 
Let $G=SL_n(\bc)$. For a $\bc$-algebra ${\mathcal R}$
let $G({\mathcal R})$ be the set of ${\mathcal R}$--rational points of $G$, i.e., the
set of algebra homomorphisms from the coordinate ring $\bc[G]\rightarrow {\mathcal R}$.
Then $G({\mathcal R})$ comes naturally equipped again with a group structure. In our case
we can identify $SL_n({\mathcal R})$ with the set of
$n\times n$--matrices with entries in ${\mathcal R}$ and determinant 1.
 
Denote $\co=\bc[[t]]$ the ring of formal power series in one variable and let $\ck=\bc((t))$ be 
its fraction field, the field of formal Laurent series. Denote $v:\ck^*\rightarrow\bz$
the standard valuation on $\ck$ such that $\co=\{f\in\ck\mid v(f)\ge 0\}$. 
The {\it loop group} $G(\ck)$ is the set of 
$\ck$--valued  points of $G$, we denote by $G(\co)$ its subgroup of $\co$--valued points.
The latter has a decomposition as a semi-direct product $G\semi G^1(\co)$, where
we view $G\subset G(\co)$ as the subgroup of constant loops and $G^1(\co)$ is the
subgroup of elements congruent to the identity modulo $t$. Note that we can describe
$G^1(\co)$ also as the image of ${\mathfrak{sl}}_n\otimes_\bc t\bc[[t]]$ via the exponential map 
(where ${\mathfrak{sl}}_n=\Lie G$).

The {\it rotation operation} $\gamma:\bc^*\rightarrow \Aut(\ck)$, 
$\gamma(z)\big(f(t)\big)=f(zt)$ gives rise to group automorphisms
$\gamma_G:\bc^*\rightarrow \Aut(G(\ck))$, we denote $\cl(G(\ck))$ 
the semidirect product $\bc^*\semi G(\ck)$.  The rotation operation 
on $\ck$ restricts to an operation on $\co$ and hence we have a natural
subgroup $\cl(G(\co)):=\bc^*\semi G(\co)$ (for this and the following see \cite{Ku},
Chapter 13).

Let $\hat\cl(G)$ be the affine Kac-Moody group associated to
the affine Kac--Moody algebra 
$$
\hat\cl(\Lg)=\Lg\otimes \ck\oplus\bc c\oplus \bc d,
$$
where $0\rightarrow \bc c\rightarrow \Lg\otimes \ck\oplus\bc c\rightarrow
\Lg\otimes \ck\rightarrow 0$ is the universal central extension of the
{\it loop algebra} $ \Lg\otimes \ck$ and $d$ denotes the scaling element.
We have corresponding exact sequences also on the level of groups,
i.e., $\hat\cl(G)$ is a central extension of $\cl(G(\ck))$ 
$$
1\rightarrow\bc^*\rightarrow \hat\cl(G)\mapright{pr} \cl(G(\ck))\rightarrow 1
$$
(see~\cite{Ku}, Chapter 13). 

Fix as {\it maximal torus} $T\subset G$ the diagonal matrices and fix as {\it Borel subgroup} $B$
the upper triangular matrices in $G$. Denote by $B^-\subset G$ the Borel subgroup  of lower triangular matrices. 
We denote $\la\cdot,\cdot\ra$ the non--degenerate pairing between 
the {\it character group} $\Mor(T,\bc^*)$ of $T$ and the group $\Mor(\bc^*,T)$ of 
{\it cocharacters}. We identify $\Mor(\bc^*,T)$ with the quotient 
$T(\ck)/T(\co)$, so we use the same symbol $\lam^\vee$ for the cocharacter
and the point in $\cg=G(\ck)/G(\co)$. 

Let $N=N_G(T)$ be the normalizer in $G$ of the fixed maximal torus $T\subset G$, we denote
by $W$ the {\it Weyl group} $N/T$ of $G$. Let $N_\ck$ be the subgroup of $G(\ck)$ 
generated by $N$ and $T(\ck)$, the affine Weyl group is defined as $\Waff=N_\ck/T\simeq W\semi \Mor(\bc^*,T)$.
\section{Roots, weights and coweights}
We use for the root system and the (abstract) weights and coweights the same notation as in \cite{Bourb}: Let  $\R^{n}$ be the
real vector space equipped with the canonical basis $\epsilon_1,\ldots,\epsilon_n$ and a scalar product $(\cdot,\cdot)$
such that the latter basis is an orthonormal basis. Denote by $\A\subset\R^n$ the subspace orthogonal to $\epsilon_1+\ldots+\epsilon_n$. 
The root system $\Phi\subset \A$ and the coroot system $\Phi^\vee$
can be identified with each other:
$$
\Phi=\{\pm(\epsilon_i-\epsilon_j)\mid 1\le i<j\le n\}=\{\alpha^\vee=\frac{2\alpha}{(\alpha,\alpha)}\mid\alpha\in\Phi\}=\Phi^\vee.
$$
Hence we can also identify the root lattice $R$ and the coroot lattice $R^\vee$. Similarly, we can identify the abstract
weight lattice $X=\{\mu\in\A\mid  \forall \alpha^\vee\in\Phi^\vee: (\mu,\alpha^\vee)\in\bz\}$ and the abstract coweight lattice
$X^\vee=\{\mu\in\A\mid (\mu,\alpha)\in\bz\ \forall \alpha\in\Phi\}$. To avoid confusion between weights
and coweights, we write ${\lambda}\in X$ for the (abstract) weights and ${\lambda}^\vee\in X^\vee$ instead of $\lambda\in X^\vee$
 for the (abstract) coweights.

For the group $G=SL_n(\bc)$ we have $X=\Mor(T,\bc^*)$, so the character group is the full abstract weight lattice, and 
the group of cocharacters is the coroot lattice $R^\vee=\Mor(\bc^*,T)$. For the adjoint group $G'=PSL_n(\bc)$ let $p:SL_n(\bc)\rightarrow G'$
be the isogeny and set $T'=p(T)$. We have in this case
$R=\Mor(T',\bc^*)$, so the character group is the root lattice, and 
the group of cocharacters is the full lattice of abstract coweights $X^\vee=\Mor(\bc^*,T)$.

A basis of the root system is $\Pi=\{\alpha_i=\epsilon_i-\epsilon_{i+1}\mid i=1,\ldots,n-1\}$, 
let $\Phi^+$ be the corresponding set of positive roots. We often identify $\A$ with
$\R^{n}/\R(\epsilon_1+\ldots+\epsilon_n)$, the description for the fundamental weights as 
$\omega_i=\epsilon_1+\ldots+\epsilon_i$, $i=1,\ldots,n-1$ is to be understood with respect to this identification. The abbreviation
$\epsilon_\ui=\epsilon_{i_1}+\ldots+\epsilon_{i_d}$ and $\epsilon^\vee_\ui=(\epsilon_{i_1}+\ldots+\epsilon_{i_d})^\vee$
for $\ui\in I_{d,n}$ is meant in the same way and provides a natural bijection between elements of $I_{d,n}$ and 
Weyl group conjugates of $\omega_d$ respectively $\omega_d^\vee$.

Given a root and coroot datum, we have the hyperplane arrangement defined by the set 
$\ens{(\alpha,n)\ |\ \alpha \in \Phi, n \in \Z}$ of affine roots. The couple $(\alpha,n)$ 
corresponds to the real affine root $\alpha + n\delta$ with $\delta$ being the smallest positive imaginary root
 of the affine Kac-Moody algebra $\hat\cl(\Lg)$. 
We associate to an affine root $(\alpha,n)$ the affine reflection 
$s_{\alpha,n}:x^\vee \mapsto x^\vee- (\left( \alpha,x^{\vee} \right)+ n)\alpha^{\vee}$ for $x^\vee\in\A$,
and the affine hyperplane $H_{\alpha,n} = \ens{x^\vee \in \A\ |\ \left( \alpha,x^{\vee} \right) + n = 0}$. We write
\[ H_{\alpha,n}^+ = \ens{x^\vee \in \A\ |\ \left( \alpha,x^{\vee} \right) + n \geq 0} \]
for the corresponding closed positive half-space and analogously 
\[ H_{\alpha,n}^- = \ens{x^\vee \in \A\ |\ \left( \alpha,x^{\vee} \right) + n \leq 0}\]
for the negative half space.
\begin{defn}\rm
Given a gallery $\gamma=(\ui_1,\ldots,\ui_r)$ of  type $\ud$, the associated
{\it polyline of type $(d_1,\ldots,d_r)$} in $\A$ is the sequence of coweights 
$\varphi(\gamma)=(\mu^\vee_0,\mu^\vee_1,\ldots,\mu_r^\vee)$
defined by
$$
\mu^\vee_0=0,\ \mu^\vee_1=\epsilon^\vee_{\ui_1},\ \mu^\vee_2=\epsilon^\vee_{\ui_1}+\epsilon^\vee_{\ui_2},
\ \mu^\vee_2=\epsilon^\vee_{\ui_1}+\epsilon^\vee_{\ui_2}+\epsilon^\vee_{\ui_3}, \ldots,\mu^\vee_r=\epsilon^\vee_{\ui_1}+\ldots+\epsilon^\vee_{\ui_r}.
$$
We often identify the sequence with the
polyline joining successively the origin $0$ with $\mu^\vee_1$, $\mu^\vee_1$ with $\mu^\vee_2$, $\mu^\vee_2$ with $\mu^\vee_3$ etc.
The last coweight $\mu^\vee_r$ is called the coweight of the gallery.
\end{defn}
The Weyl group $W$ is the finite subgroup of $\Gl(\A)$ generated by the reflections $s_{\alpha,0}$ for $ \alpha \in \Phi$, 
the affine Weyl group $W^{\af}$ is the group of affine transformations of $\A$ generated by the affine reflections 
$s_{\alpha,n}$ for $ (\alpha,n) \in \Phi \times \Z$. 

A fundamental domain for the action of $W$ on $\A$ is given by the dominant Weyl chamber
$C^+ = \ens{ x^\vee \in \A\ |\ \left( \alpha,x^{\vee} \right) \geq 0,\ \forall \alpha \in \Phi^+}$.
Similarly, the fundamental alcove $\Delta_f = \ens{x^\vee \in \A\ |\ 0 \leq \left( \alpha,x^{\vee} \right) \leq 1,\ \forall \alpha \in \Phi^+}$
is a fundamental domain for the action of $W^{\af}$ on $\A$. 

\section{The affine Gra\ss mann variety}\label{affgrass}

Let $\hat\cl(G)$ be the affine Kac-Moody group as in section~\ref{loopgroup}.
Recall the exact sequences 
$$
1\rightarrow\bc^*\rightarrow \hat\cl(G)\mapright{pr} \cl(G(\ck))\rightarrow 1
$$
(see~\cite{Ku}, Chapter 13). Denote $P_0\subset \hat\cl(G)$ the maximal parabolic subgroup $pr^{-1}(\cl(G(\co)))$.
We have four incarnations of the affine Grassmann variety:
\begin{equation}\label{fourgrassmann}
\cg=G(\ck)/G(\co)=\cl(G(\ck))/\cl(G(\co))=\hat\cl(G)/\cp_0 = G(\bc[t,t^{-1}])/G(\bc[t]).
\end{equation} 
Note that $G(\ck)$ and $\cg$ are {\it ind}--schemes and $G(\co)$ is a group 
scheme (see \cite{Ku}, \cite{Lu}). 

Let now $G'$ be a simple complex algebraic group with the same Lie algebra ${\mathfrak{sl}}_n$ as $G$ and
let $p:G\rightarrow G'$ be an isogeny with $G'$ being simply connected. Then
$G'(\co)\simeq G'\semi (G)^1(\co)$, the natural map $p_\co:G(\co)\rightarrow G'(\co)$ is surjective
and has the same kernel as $p$. Let $T'$ be a maximal torus such that $p(T)=T'$ and consider the character group 
$\Mor(T',\bc^*)$ the group $\Mor(\bc^*,T')$ of cocharacters.
The map $p:T\rightarrow T'$ induces an inclusion $\Mor(\bc^*,T)\hookrightarrow \Mor(\bc^*,T')$. 

In particular, if we start with $PSL_n(\bc)$ (instead of $SL_n(\bc)$), then $\Mor(T',\bc^*)$ is the root lattice $R$ and
we can identify $\Mor(\bc^*,T')$ with the full coweight lattice $X^\vee$:
$$
X^\vee=\Mor(\bc^*,T')=\bz\omega^\vee_1\oplus\ldots \oplus \bz \omega^\vee_{n-1}.
$$
The quotient  $\Mor(\bc^*,T')/\Mor(\bc^*,T)$ measures the difference between $\cg$ and the affine
Gra\ss mann variety $\cg'=G'(\ck)/G'(\co)$. In fact, $\cg$ is connected, and the connected 
components of $\cg'$ are indexed by $\Mor(\bc^*,T')/\Mor(\bc^*,T)$. The natural maps 
$p_\ck:G(\ck)\rightarrow G'(\ck)$ and $p_\co:G(\co)\rightarrow G'(\co)$ induce a 
$G(\ck)$--equivariant inclusion $\cg\hookrightarrow \cg'$, which is an isomorphism onto the 
component of $\cg$ containing the class of $1$. Now $G(\ck)$ acts via $p_\ck$ on
all of $\cg'$, and each connected component is a homogeneous space for $G(\ck)$,
isomorphic to $G(\ck)/\cq$ for some parahoric subgroup $\cq$ of $G(\ck)$ which is
conjugate to $G(\co)$ by an (outer) automorphism. In the same way the action of $\hat\cl(G)$
on $\cg$ extends to an action on $\cg'$, each connected component being a homogeneous space. 

Let $ev:G'(\co)\rightarrow G'$ be the evaluation maps at $t=0$, and let $\cb_\co=ev^{-1}(B)$ be the corresponding Iwahori subgroup.
Denote by $\Bc$ the corresponding Borel subgroup of $\hat\cl(G)$.
Recall the orbit decomposition (see for example \cite{Ku}):
$$
\cg'=\bigcup_{\mu^\vee\in {X^\vee}^+} G(\co) \mu^\vee
$$
where $X^{\vee+}$ denotes the set of dominant (abstract) coweights. 

Throughout the remaining part of the paper, we study for $G=SL_n(\bc)$ the action of the groups
$G(\co)$, $G(\ck)$ (and other relevant subgroups of $G(\ck)$) on the affine Gra\ss mann variety $\cg'$ for $G'=PSL_n$.

For a dominant coweight $\lam^\vee\in {X^\vee}^+$ denote by $X_{\lam^\vee}$
the $G(\co)$-stable Schubert variety 
$$
X_{\lam^\vee}=\overline{G(\co) \lam^\vee}\subseteq \cg'.
$$

\section{Galleries and subgroups}
Let $\gamma=(\ui_1,\ldots,\ui_r)$ be a gallery of  type $\ud$ and denote  
$\varphi(\gamma)=(\mu^\vee_0,\mu^\vee_1,\ldots,\mu_r^\vee)$ the associated polyline.
Let $[\mu^\vee_{j-1},\mu^\vee_{j}]$ be the convex hull of the two vertices. 
We associate to the gallery and the vertices on the polyline some subgroups: 

Denote $P_{\mu^\vee_j}$ the
parabolic subgroup of $\hat\cl(G)$ containing $\bc^*\semi T$ and the root subgroups associated to the real roots
\begin{equation}\label{PhiPj}
{\hat{\Phi}}_{j}=\{(\alpha,n)\mid \mu_j^\vee\in H_{\alpha,n}^+ \}, \ j=0,\ldots,r.
\end{equation}
Denote $Q_{\mu^\vee_j}$ the parabolic subgroups of $\hat\cl(G)$ containing $\bc^*\semi T$ 
and the root subgroups associated to the real roots
\begin{equation}\label{PhiQj}
\hat \Phi_{j,j+1}=\{(\alpha,n)\mid [\mu^\vee_{j},\mu^\vee_{j+1}]\subset H_{\alpha,n}^+ \},\  j=0,\ldots,r-1,
\end{equation}
Note that $P_{\mu^\vee_{j}}\supset Q_{\mu^\vee_j}\subset P_{\mu^\vee_{j+1}}$. Another important subgroup of $P_{\mu^\vee_{j}}$ is 
generated by its root subgroups corresponding to roots which are negative in the classical sense:
\begin{equation}\label{PhiUj}
U^-_{\mu^\vee_j}=\langle U_{(\beta,m)}\mid \beta\prec 0, (\beta,m)\in {\hat{\Phi}}_{j}\rangle \subset P_{\mu^\vee_{j}},\  j=0,\ldots,r.
\end{equation}
The following set of roots is relevant later for the description of certain cells. Let
$$
\Psi_{\gamma,j}=\{(\alpha,n)\mid \alpha\succ0, \mu^\vee_{j}\in H_{\alpha,n}, 
[\mu^\vee_{j},\mu^\vee_{j+1}]\not\subseteq H_{\alpha,n}^- \}
$$
and set
$$
{\mathbb U}^-_{\gamma,j}=\prod_{(\alpha,n)\in \Psi_{\gamma,j}} U_{(-\alpha,-n)} \subset U^-_{\mu^\vee_{j}}.
$$
We associate to the gallery the product of root subgroups
$$
\bU^-_{\gamma}= \bU^-_{\gamma,0}\times \bU^-_{\gamma,1}\times\cdots\times \bU^-_{\gamma,r-1}.
$$
\section{Galleries and Bott-Samelson varieties}\label{galleryandBottSamelson}
For a gallery $\gamma=({\ui_1},{\ui_2},\ldots,{\ui_r})$ of type $\ud=(d_1,\ldots,d_r)$
let $\varphi(\gamma)=(\mu^\vee_0,\ldots,\mu_r^\vee)$ be the associated polyline.
Recall that $P_{\mu^\vee_{j}}\supset Q_{\mu^\vee_j}\subset P_{\mu^\vee_{j+1}}$, so we can associate to $\gamma$ the Bott-Samelson variety
\begin{equation}\label{BottSamelson1}
\Sigma_\ud= P_{\mu^\vee_0}\times_{Q_{\mu^\vee_0}} P_{\mu^\vee_1}\times_{Q_{\mu^\vee_1}}\times \ldots
\times_{Q_{\mu^\vee_{r-1}}} 
P_{\mu^\vee_{r}}/P_{\mu^\vee_{r}},
\end{equation}
which is defined as the quotient of $P_{\mu^\vee_0}\times P_{\mu^\vee_1}\times  \ldots \times P_{\mu^\vee_{r}}$
by $Q_{\mu^\vee_0} \times Q_{\mu^\vee_1}\times \ldots \times Q_{\mu^\vee_{r-1}}\times P_{\mu^\vee_{r}}$
with respect to the action given by:
$$
(q_0,q_1,..., q_{r-1}, q_r)\circ (p_0,p_1,...,p_r)  = (p_0q_0,q_0^{-1}p_1q_1,..., q_{r-1}^{-1}p_r q_r).
$$
The fibred product $\Sigma_\ud$ is a smooth projective complex variety, its points are denoted by $[p_0,\ldots,p_r]$. 

Corresponding to the parabolic subgroups $Q_{\mu^\vee_j}, P_{\mu^\vee_{j}}$ of $\hat\cl(G)$ one has the parahoric subgroups
$\Pc_{\mu^\vee_j},\Qc_{\mu^\vee_j}$ of $G(\ck)$ such that $P_{\mu^\vee_{j}}=pr^{-1}(\bc^*\semi\Pc_{\mu^\vee_j})$ and 
$Q_{\mu^\vee_{j}}=pr^{-1}(\bc^*\semi\Qc_{\mu^\vee_j})$. We get another description of $\Sigma_\ud$ as
\begin{equation}\label{BottSamelson2}
\Sigma_\ud= \Pc_{\mu^\vee_0}\times_{\Qc_{\mu^\vee_0}} \Pc_{\mu^\vee_1}\times_{\Qc_{\mu^\vee_1}}\times \ldots\times_{\Qc_{\mu^\vee_{r-1}}} \Pc_{\mu^\vee_{r}}/\Pc_{\mu^\vee_{r}}.
\end{equation}
Set ${\mathcal F}_\gamma=\hat\cl(G)/P_{\mu^\vee_0}
\times \hat\cl(G)/P_{\mu^\vee_{1}}\times\cdots\times \hat\cl(G)/P_{\mu^\vee_{r}}$ and consider the 
natural morphism $\iota:\Sigma_\ud\rightarrow {\mathcal F}_\gamma$ given by sending the class $[p_0,p_1,\ldots,p_r]$ to the sequence
$$
(P_{\mu^\vee_0}, p_0Q_{\mu^\vee_0}p_0^{-1}, p_0P_{\mu^\vee_1}p_0^{-1}, p_0p_1Q_{\mu^\vee_1}p_1^{-1}p_0^{-1}, \ldots,
p_0\cdots p_{r-1}P_r p^{-1}_{r-1}\cdots p^{-1}_0)\in {\mathcal F}_\gamma.
$$
The image of $\Sigma_\ud$ can be described as the set of sequences of parabolic subgroups (see for example \cite{GL2}) 
satisfying the following inclusion relations:
\begin{equation}\label{BottSamelson3}
\left\{(P_{\mu^\vee_0},Q_0,\ldots,P_r)\mid 
\begin{array}{l}
P_{\mu^\vee_0}\supset Q_0\subset P_1\supset \ldots \subset P_r\\
 P_j\text{\ is conjugate to $P_{\mu^\vee_{j}}$,}\forall j=1,\ldots,r\\
 Q_j\text{\ is conjugate to $Q_{\mu^\vee_{j}}$}\forall j=0,\ldots,r-1.
\end{array}\right\}
\end{equation}
The morphism $\iota$ is an isomorphism onto the image.
If $\gamma'$ is a gallery of the same type as $\gamma$ with associated polyline $\varphi(\gamma')=({\mu'}^\vee_0,\ldots,{\mu'}_r^\vee)$,
then $\mu^\vee_0={\mu'}^\vee_0=0$, ${\mu'}^\vee_1$ is conjugate to ${\mu}^\vee_1$ by an element $\tau_0$ in the Weyl group of $P_{\mu_0}$ and hence
$P_{\mu'_1}=\tau_0P_{\mu_1}\tau_0^{-1}$,
${\mu'}^\vee_2$ is conjugate to $\tau_0({\mu}^\vee_2)$ by an element $\tau_1$ in the Weyl group of $P_{\mu'_1}=\tau_0P_{\mu_1}\tau_0^{-1}$
and hence $P_{\mu'_2}=\tau_1\tau_0 P_{\mu_2}\tau_0^{-1}\tau_1^{-1}$ etc.

So $\gamma$ and $\gamma'$ define in \eqref{BottSamelson3} the same variety and hence $\Sigma_\ud$ depends only on the type $\ud$ and not on the gallery. 
We say that we consider $\Sigma_\ud$ with center $\gamma$
if we want to emphasize that the  parabolic subgroups in \eqref{BottSamelson1} correspond to the coweights lying on the polyline associated to $\gamma$.

For a fixed type $\ud$ let $\T_0$ be the unique key tableau of column shape $\ud$ having as filling only $1$'s in the top row, $2$'s in the second row etc. Denote by $\gamma_0$ the associated gallery and let $\varphi(\gamma_0)=(\mu^\vee_0,\ldots,\mu^\vee_r)$ be the associated polyline. 
The variety $\Sigma_\ud$ is by \eqref{BottSamelson2} naturally endowed with a $T$-action. 
\begin{lem}\label{lemma1}
We consider $\Sigma_\ud$ with center $\gamma_0$. There exists a natural bijection between $T$-fixed points in 
$\Sigma_\ud$ and galleries of type $\ud$ such that 
the gallery $\gamma_0$ corresponds to the class $[1,\ldots,1]$ in 
$P_{\mu^\vee_0}\times_{Q_{\mu^\vee_0}} \ldots\times_{Q_{\mu^\vee_{r-1}}} P_{\mu^\vee_{r}}/P_{\mu^\vee_{r}}$.
\end{lem}
\begin{proof} We have a sequence of projections of the fibered spaces:
\begin{equation}
\label{ fibrations}
\Sigma_\ud\mapri{\phi_{r-1}} P_{\mu^\vee_0}\times_{Q_{\mu^\vee_0}} \ldots\times_{Q_{\mu^\vee_{r-3}}} P_{\mu^\vee_{r-2}}/Q_{\mu^\vee_{r-2}}
\mapri{\phi_{r-2}}\cdots\mapri{\phi_{2}} P_{\mu^\vee_0}\times_{Q_{\mu^\vee_0}}P_{\mu^\vee_1}/{Q_{\mu^\vee_1}}\mapri{\phi_{1}} P_{\mu^\vee_0}/{Q_{\mu^\vee_0}}
\end{equation}
such that the fibres of $\phi_j$ are isomorphic to $P_{\mu^\vee_j}/{Q_{\mu^\vee_{j}}}$.
The fibrations are locally trivial in the Zariski topology and equivariant with respect to the $T$-action.
It is now easy to see that the $T$-fixed points in $\Sigma_\ud$ are in a natural one-to one correspondence
with $T$-fixed points in 
$$
P_{\mu^\vee_0}/Q_{\mu^\vee_0}\times P_{\mu^\vee_1}/Q_{\mu^\vee_1}\times \ldots \times P_{\mu^\vee_{r-1}}/Q_{\mu^\vee_{r-1}}.
$$
Let  $L_ {\mu^\vee_j}\subset P_{\mu^\vee_j}$ be the Levi subgroup containing $\bc^*\semi T$, then the semisimple
part $SL_ {\mu^\vee_j}$ of $L_ {\mu^\vee_j}$ is isomorphic to $SL_n(\bc)$. With respect to this isomorphism, 
the intersection $P=SL_ {\mu^\vee_j}\cap Q_{\mu^\vee_j}$ is a maximal parabolic subgroup of $SL_ {\mu^\vee_j}$ such that
$$
P_{\mu^\vee_j}/Q_{\mu^\vee_j}\simeq SL_ {\mu^\vee_j}/P\simeq SL_n(\bc)/P_{\omega^\vee_{d_j}}.
$$
The $T$-fixed points of this variety are naturally indexed by elements in $I_{d_{j},n}$, $j=1,\ldots,r$,
proving the claim.  
\end{proof}

\section{MV-cycles and galleries}\label{mvgallery}
We fix a type $\ud=(d_1,\ldots,d_r)$. Since the Bott-Samelson variety $\Sigma_\ud$
depends only on the type and not on the choice of gallery, we fix during this section for convenience $\gamma$
such that the associated polyline is wrapped around the fundamental alcove, i.e. if $\gamma=(\ui_1,\ldots,\ui_r)$
and $\varphi(\gamma)=(\mu^\vee_0=0,\mu^\vee_1,\ldots, \mu^\vee_r)$, then all the coweights $\mu^\vee_0,\mu^\vee_1,\ldots, \mu^\vee_r$
are vertices of the fundamental alcove. Since the type of the gallery is fixed, there exists only one gallery of this type with this property. 

Set $\lambda^\vee=\sum_{j=1}^r\omega_{d_j}^\vee$ and consider the Bott-Samelson variety of center $\gamma$:
$$
\Sigma_\ud= P_{\mu^\vee_0}\times_{Q_{\mu^\vee_0}} P_{\mu^\vee_1}\times_{Q_{\mu^\vee_1}}\times \ldots\times_{Q_{\mu^\vee_{r-1}}} P_{\mu_r^\vee}/P_{\mu_r^\vee}.
$$
Note that for this choice of $\gamma$ all the associated parabolic subgroups are standard parabolic subgroups, i.e. $\Bc\subset P_{\mu^\vee_j}, Q_{\mu^\vee_j}$.

The last parabolic subgroup $P_{\mu_r^\vee}$ in the fibered product corresponds to a maximal standard parabolic subgroup of $\hat\cl(G)$.
In $\cg'$, this group corresponds to the point $\mu_r^\vee$
(i.e. $P_{\mu_r^\vee}\subset \hat\cl(G)$ is the stabilizer of this point with respect to the action of $\hat\cl(G)$ on $\cg'$), so we get a canonical map
$$
\begin{array}{rcccl}
\pi&:&\Sigma_\ud= P_{\mu^\vee_0}\times_{Q_{\mu^\vee_0}} P_{\mu^\vee_1}\times_{Q_{\mu^\vee_1}}\times \ldots\times_{Q_{\mu^\vee_{r-1}}} P_{\mu_r^\vee}/P_{\mu_r^\vee}
&\rightarrow &\cg'\\
&&[p_0,\ldots,p_{r-1},p_r]&\mapsto&p_0\cdots p_{r-1}\mu_r^\vee.
\end{array}
$$
It is well known that the image $\pi(\Sigma_{\ud})$ is the Schubert variety $X_{\lambda^\vee}$, 
and the map is in fact a desingularization (see for example \cite{GL2}).

Let $\eta:\bc^*\rightarrow T$ be a generic anti-dominant one parameter subgroup. $\Sigma_\ud$ is a smooth
projective variety endowed with a $T$-action and hence a $\eta$-action. Since $\eta$ is generic,
the $T$-fixed points are the same as the $\eta$-fixed points, so the latter are naturally indexed by galleries of type $\ud$
(see Lemma~\ref{lemma1}). 
To each fixed point $y_\zeta$ corresponding to a gallery $\zeta$
of type $\ud$, we can associate the Bia{\l}ynicki-Birula cell \cite{BB}:
$$
C_\zeta=\{x\in\Sigma_\ud\mid \lim_{t\rightarrow 0} \eta(t)x=\zeta\}.
$$
Recall that $\Sigma_\ud$ is the disjoint union $\bigcup C_\zeta$, where $\zeta$ is running over the set of all galleries of type $\ud$.
 
The Schubert variety $X_{\lam^\vee}$ has a similar decomposition, but which is in general not given by cells.
Let $U^-$ be the unipotent radical of $B^-$. Given a coweight $\mu^\vee$, the closure of an irreducible component of
$$
({U^-(\ck).\mu^\vee)\cap (G(\co).{\lambda^\vee}})\subset \cg'
$$
is called a Mirkovi\'c-Vilonen cycle, or short just MV-cycle \cite{MV} of coweight $\mu^\vee$ in $X_{\lambda^\vee}$. It is known that the intersection is nonempty
if and only if in the finite dimensional irreducible representation of $SL_n(\bc)$ (viewed as Langlands dual group of $PSL_n(\bc)$) 
of highest weight $\lambda^\vee$ the weight space corresponding to $\mu^\vee$ is non-zero. Further, the number of irreducible
components is equal to the dimension of this weight space.
One has a disjoint union
$$
X_{\lam^\vee}=\bigcup_{\mu^\vee\in X^\vee}( X_{\lam^\vee}\cap U^-(\ck).\mu^\vee),
$$ 
and since $\lim_{t\rightarrow 0} \eta(t)u\eta(t)^{-1}=1$ for all $u\in U^-(\ck)$, it follows 
that 
$$
X_{\lam^\vee}\cap U^-(\ck).\mu^\vee=\{y\in X_{\lam^\vee}\mid  \lim_{t\rightarrow 0} \eta(t)y=\mu^\vee\}.
$$
The desingularization map $\pi:\Sigma_\ud\rightarrow X_{\lambda^\vee}$ is $T$-equivariant, hence if $\mu^\vee$ is the coweight of a gallery $\gamma$ of type $\ud$,
then the image $\pi(C_\zeta)$ of the cell is contained in $X_{\lam^\vee}\cap U^-(\ck).\mu^\vee$. Since
the image is irreducible, it is contained in some MV-cycle. 
\section{Main Theorem}
Let $\gamma$ be a gallery of type $\ud$ with associated polyline $\varphi(\gamma)=(\mu_0^\vee,\ldots,\mu_r^\vee)$, 
set $\lam^\vee=\sum_{j=1}^r \omega_{d_j}^\vee$ and denote by
$\pi_{\ud}:\Sigma_\ud\rightarrow X_{\lam^\vee}$ the desingularization map. Let $\T$ be the unique
semistandard Young tableau such that the words $w_\gamma$ and $w_\T$ are Knuth equivalent.
Let $\underline{\mathbf c}=(c_1,\ldots,c_s)$ be the shape of  $\T$, set $\nu^\vee=\sum_{j=1}^s \omega_{c_j}^\vee$
and let $X_{\nu^\vee}\subseteq X_{\lam^\vee}$ be the corresponding Schubert variety.
\begin{thm}\label{mainthm1}
\begin{itemize}
\item[{\it a)}]The closure $\overline{\pi_{\ud}(C_\gamma)}\subseteq X_{\lam^\vee}$ is a MV-cycle for the (possibly smaller) Schubert variety 
$X_{\nu^\vee}\subseteq X_{\lam^\vee}$. The weight of the MV-cycle is $\mu_r^\vee$.
\item[{\it b)}]
Given a second gallery $\gamma'$ of type $\ud'$, set ${\lam'}^\vee=\sum_{j=1}^r \omega_{d'_j}^\vee$
and let $\pi_{\ud'}:\Sigma_{\ud'}\rightarrow X_{{\lam'}^\vee}$ be the desingularization map. The following conditions are equivalent:
\begin{itemize}
\item[{\it i)}] $\gamma$ and $\gamma'$ are Knuth equivalent.
\item[{\it ii)}] $\overline{\pi_{\ud}(C_\gamma)}=\overline{\pi_{\ud'}(C_{\gamma'})}$
\end{itemize}
\end{itemize}
\end{thm}
As a consequence we get a geometric interpretation of the Knuth relations:
\begin{coro}
For  two words $w_1,w_2$ of the same length $N$ in the alphabet $\bA$ the following conditions are equivalent:
\begin{itemize}
\item[${\it i)}$] $w_1$ and $w_2$ are Knuth equivalent
\item[${\it ii)}$] $\overline{\pi(C_{w_1})}=\overline{\pi(C_{w_2})}$, where $\pi$ is the canonical map $\pi:\Sigma_{N\epsilon_1^\vee}\rightarrow X_{N\epsilon_1^\vee}$,
and the words $w_1,w_2$ are viewed as galleries of type $\underbrace{(1,1,\ldots,1)}_N$. 
\end{itemize}
\end{coro}
The proof will be given in section~\ref{proffofmainthm}.

\section{The cell $C_\gamma$}
We recall a description of the cell and determine its dimension. 
For a gallery $\gamma=({\ui_1},{\ui_2},\ldots,{\ui_r})$ of type $\ud=(d_1,\ldots,d_r)$ set $\lam^\vee=\sum_{j=1}^r \omega_{d_j}^\vee$.
Let $\varphi(\gamma)=(\mu^\vee_0,\mu^\vee_1,\ldots, \mu^\vee_r)$ be the associated polyline
and denote by $\Sigma_\ud$ the associated Bott-Samelson variety of center $\gamma$.
The coweights $\mu^\vee_{j}$ and $\mu^\vee_{j+1}$ are joined
by the segment $[\mu^\vee_{j},\mu^\vee_{j+1}]$, $j=0,\ldots,r-1$. 
Let $H_{\alpha,n}$, $\alpha>0$, be a hyperplane such that $\mu^\vee_{j}\in H_{\alpha,n}$.
We say that $\gamma$ crosses the wall $H_{\alpha,n}$ in $\mu^\vee_{j}$ in 
the positive (negative) direction if  $[\mu^\vee_{j},\mu^\vee_{j+1}]\not\subset H^-_{\alpha,n}$
(respectively $[\mu^\vee_{j},\mu^\vee_{j+1}]\not\subset H^+_{\alpha,n}$).  Set
$$
\begin{array}{rcl}
\sharp^+(\gamma)&=&\sum_{j=0}^{r-1} (\sharp\text{\ positive wall crossings in $\mu^\vee_j$})=\sum_{j=0}^{r-1} \vert \Psi_{\gamma,j}\vert;\\
\sharp^-(\gamma)&=&\sum_{j=0}^{r-1} (\sharp\text{\ negative wall crossings in $\mu^\vee_j$});\\
\sharp^\pm(\gamma)&=&\sharp^+(\gamma)+\sharp^-(\gamma).
\end{array}
$$
The value of the sum $\sharp^\pm(\gamma)$ depends, by definition, only on the type of the gallery and not on the actual
choice of the gallery. It is easy to check that
$$
\sharp^\pm(\gamma)= (\lambda^\vee,2\rho^\vee),
$$
where $\rho^\vee$ is the sum of the fundamental coweights.

The following proposition is proved in \cite{GL2}, Proposition 4.19.
\begin{prop}\label{celldescrition}
The map 
$$
\begin{array}{rcl}
\bU^-_{\gamma,0}\times \bU^-_{\gamma,1}\times \bU^-_{\gamma,2}\times \ldots \times \bU^-_{\gamma,{r-1}} &\rightarrow &\Sigma_{\ud},\\
(u_0,\ldots, u_{r-1})&\mapsto& 
[u_0,\ldots,u_{r-1},1]
\end{array},
$$
is an isomorphism onto the cell $C_\gamma$. In particular, $\sharp^+(\gamma)$ is the dimension of the cell.
\end{prop}
The image of the map 
$$
\begin{array}{rcl}
\chi:U^-_{\mu_0^\vee}\times U^-_{\mu_1^\vee}\times \ldots \times U^-_{\mu_{r-1}^\vee} &\rightarrow &\Sigma_{\ud},\\
(u_0,\ldots, u_{r-1})&\mapsto& 
[u_0,\ldots,u_{r-1},1]
\end{array},
$$
is obviously (by the choice of $\eta$) a subset of $C_\gamma$. But $\bU^-_{\gamma,j}$ is a subset of the group $U^-_{\mu_j^\vee}$, 
so we get as an immediate consequence:
\begin{coro}
The map $\chi$ has as image the cell $C_\gamma$.
\end{coro}
The description of the cell in the proposition above has as an immediate consequence:
\begin{coro}\label{coro1image}
$\pi(C_\gamma)= \bU^-_{\gamma,0} \bU^-_{\gamma,1} \bU^-_{\gamma,2} \cdots  \bU^-_{\gamma,{r-1}}.\mu_r^\vee=
U^-_{\mu_0^\vee} U^-_{\mu_1^\vee} \cdots  U^-_{\mu_{r-1}^\vee}.\mu_r^\vee$.
\end{coro}
\begin{lem}\label{dimension}
The dimension of the cell $C_\gamma$ depends only on the coweight of the gallery. More precisely,
$\dim C_\gamma=\sharp^+(\gamma)=(\lam^\vee +\mu_r^\vee,\rho^\vee)$.
\end{lem}
\begin{proof}
The proof is by induction on the height of $\lam^\vee-\mu_r^\vee$. If $\lam^\vee=\mu_r^\vee$, then
there exists only one gallery of this type with coweight $\lam^\vee$, it must be 
$\gamma=(\omega_{d_1}^\vee,\ldots,\omega_{d_r}^\vee)$. The hyperplanes crossed positively are
exactly all the affine hyperplanes lying between the origin and $\lam^\vee$, for each positve root
there are exactly $(\lam^\vee,\alpha^\vee)$ such hyperplanes and hence for this gallery we get
$$
\sharp^+(\gamma)=\sum_{\alpha\in\Phi^+}(\lam^\vee,\alpha^\vee)=(\lam^\vee,2\rho^\vee)=(2\lam^\vee,\rho^\vee),
$$
proving the claim in this case. Suppose now the claim holds if the height of $\lam^\vee-$ coweight of the gallery is strictly
smaller than $s$. Suppose $\gamma=(\ui_1,\ldots,\ui_r)$ is a gallery such that the height of $\lam^\vee-\mu^\vee_r$ is equal to $s$ and $s>0$.
Then there exists a simple root $\alpha$ and $1\le t\le r$ such that $\epsilon^\vee_{\ui_t}\not=\omega^\vee_{d_t}$
and  $s_\alpha(\epsilon^\vee_{\ui_t})=\epsilon^\vee_{\ui_t}+\alpha^\vee$. Let ${\ui'}_t$ be such that $\epsilon_{{\ui'}_t}=\epsilon^\vee_{\ui_t}+\alpha^\vee$
and let $\gamma'$ be the gallery
$$
\gamma'=({\ui_1},\ldots,{\ui_{t-1}},{\ui'}_t,{\ui_{t+1}},\ldots,{\ui_r}),
$$
and denote by ${\mu'}^\vee_r$ the coweight of $\gamma'$.
Then the height of $\lam^\vee-{\mu'}^\vee_r$ is equal to $s-1$ and hence 
$\sharp^+(\gamma')=(\lam^\vee+{\mu'}^\vee_r,\rho^\vee)$. It follows:
$$
\begin{array}{rcl}
\sharp^+(\gamma)=\sum_{j=1}^r  \sharp \Psi_{\gamma,j}
=(\sum_{j=1}^r  \sharp \Psi_{\gamma',j})+1
&=&(\lam^\vee+{\mu'}^\vee_r,\rho^\vee)+1\\
&=&(\lam^\vee+{\mu'}^\vee_r,\rho^\vee)+(\alpha^\vee,\rho^\vee)\\
&=&(\lam^\vee+{\mu}^\vee_r,\rho^\vee)
\end{array}
$$
\end{proof}
\section{Tail of a gallery}
Let $\gamma$ be a gallery of type $\ud$.
Let $\T_\gamma$ be the corresponding key tableau.
In Section~\ref{seWandK}, we have associated a to a key tableau a word in the alphabet $\bA$,
we just write $w_\gamma$ for the word associated to $\T_\gamma$. 

Corresponding to the type, we have the associated Bott-Samelson variety $\Sigma_\ud$, the dominant coweight $\lam^\vee=\sum \omega_{d_j}^\vee$ and the desingularization map:
$$
\pi_1:\Sigma_\ud\rightarrow X_{\lam^\vee}.
$$
And, associated to the gallery, we have the cell $C_\gamma\subset \Sigma_\ud$. Recall that the map $\gamma\rightarrow w_\gamma$ which associates to a gallery $\gamma$
the word $w_\gamma$ in the alphabet $\bA$ is far from being injective. Let us first show the following.
\begin{prop}\label{wordgleichimage}
The image $Z_\gamma = \pi_1(C_\gamma)$ depends only on the word $w_\gamma$.
\end{prop}
The proof will be given at the end of this section.
Set $Z_\gamma=\pi_1(C_\gamma)$ and note that 
$pr(U^-_{\mu_0^\vee})\subseteq U^-({\mathcal O})=U^-({\mathcal K})\cap SL_n({\mathcal O})$
by definition. 
Recall that $\Sigma_\ud$ inherits by \eqref{BottSamelson2} naturally an $SL_n({\mathcal O})$-action, and the map $\pi_1$ is equivariant with respect to this action.
The cell $C_\gamma$ is not $SL_n({\mathcal O})$-stable, but the cells
are stable under the action of $U^-({\mathcal O})$: recall that
$$
C_\gamma=\{z\in \Sigma_\ud\mid \lim_{t\rightarrow 0}\eta(t).z=\gamma\}
$$
for a fixed anti-dominant one-parameter subgroup $\eta:\bc^*\rightarrow T$. Since
$$
\lim_{t\rightarrow 0}\eta(t) u \eta(t)^{-1}=1,
$$
it follows that $u.z\in C_\gamma$ for all $z\in C_\gamma$ and all $u\in U^-({\mathcal O})$.
By the definition of $\Psi_0$ we have as an immediate consequence:
\begin{lem}\label{lemmauminus}
$Z_\gamma$ is $U^-_{\mu_0^\vee}$-stable.
\end{lem}
We want to reformulate  Lemma~\ref{lemmauminus} to make it available in a more general situation.
Let $\gamma=({\ui_1},{\ui_2},\ldots,{\ui_r})$ be a gallery of type $\ud$, denote by $\varphi(\gamma)=(\mu^\vee_0,\ldots,\mu^\vee_r)$
the associated polyline.
For $k\ge 0$ consider the tail $\gamma^{\ge k}=(\ui_k,\ldots,\ui_r)$ of the gallery $\gamma$, denote by
$\varphi(\gamma^{\ge k})=({\mu'}_0^\vee,\ldots,{\mu'}_{r-k}^\vee)$ the associated polyline and by
$\varphi^{shift}(\gamma^{\ge k})=(\mu_k^\vee,\ldots,\mu_r^\vee)$ the shifted polyline (starting in $\mu_k^\vee$ instead of $0$) 
of the tail gallery. Denote by $Z_{\gamma^{\ge k}}$ the image of the cell $C_{\gamma^{\ge k}}\subset \Sigma_{(d_k,\ldots,d_r)}$:
$$
Z_{\gamma^{\ge k}}=\bU^-_{\gamma^{\ge k},0}\cdots \bU^-_{\gamma^{\ge k},r-k-1}.{\mu'}_{r-k}^\vee 
=U^-_{{\mu'}_0^\vee}U^-_{{\mu'}_1^\vee}\cdots U^-_{{\mu'}_{r-k-1}^\vee}.{\mu'}_{r-k}^\vee
$$
and let $Z_\gamma^{\ge k}$ be the {\it truncated image}: 
$$
Z_\gamma^{\ge k}=\bU^-_{\gamma,k}\bU^-_{\gamma,k+1}\cdots \bU^-_{\gamma,r-1}.\mu_r^\vee.
$$
\begin{prop}\label{proposition3}
$Z_\gamma^{\ge k}$ is $U^-_{\mu_k^\vee}$-stable.
\end{prop}
\begin{proof}
Let $T'\subset PSL_n(\bc)$ be as in section~\ref{affgrass}.
Corresponding to a coweight $\mu^\vee_k$ let $t_{\mu_k^\vee}$ be the translation in the extended affine Weyl group. 
The translations act on the affine roots by shifts: 
$t_{{\mu_k^\vee}}:(\beta,\ell)\mapsto (\beta, \ell+ ( {\mu_k^\vee},\beta))$
and $t_{-{\mu_k^\vee}}:(\beta,k)\mapsto (\beta, \ell- ( {\mu_k^\vee},\beta))$.

Denote by $U_{\beta}:\bc((t))\rightarrow G({\mathcal K})$ the root subgroup corresponding to the root $\beta$.
Note that the conjugation by $t_{{\mu_k^\vee}}$ acts on the root subgroups by a shift: given $a\in\bc$ and $\ell\in\Z$, then 
$$
t_{\mu^\vee_k} U_\beta(at^{\ell})(t_{\mu^\vee_k})^{-1}= U_\beta(at^{\ell+( \mu_k,\beta)}),
$$
so the conjugation on the group and the translation on the coweight lattice correspond to each other.
It follows:
\begin{equation}\label{movecycle}
\begin{array}{rcl}
 t_{{\mu_k^\vee}} .Z_{\gamma^{\ge k}}&=&t_{{\mu_k^\vee}} .(\bU^-_{\gamma^{\ge k},0}\cdots \bU^-_{\gamma^{\ge k},r-k-1}.{\mu'}_{r-k}^\vee) \\
& =&\bU^-_{\gamma,k}\bU^-_{\gamma,k+1}\cdots \bU^-_{\gamma,r-1}.\mu_r^\vee\\
&=&Z_{\gamma}^{\ge k}.
\end{array}
\end{equation}
By Lemma~\ref{lemmauminus}, $Z_{\gamma^{\ge k}}$ is stabilized by $\bU^-_{\gamma,0}$,
so $ t_{{\mu_k^\vee}} \bU^-_{\gamma,0}  (t_{{\mu_k^\vee}})^{-1}= \bU^-_{\gamma,k} $ stabilzes $Z_{\gamma}^{\ge k}$. 
\end{proof}
\vskip 5pt
\noindent{\it Proof of Proposition~\ref{wordgleichimage}}.
Let $w_\gamma$ be the word (in the alphabet $\bA$) associated to $\gamma$. Since $\bA\subset\bW$,
we can view $w_\gamma$ also as a gallery of type ${\mathbf 1_N}=(1,1,\ldots,1)$,
where $N=\sum_{j=1}^r d_j$ is the length of the word. For $\gamma=(\ui_1,\ui_2,\ldots,\ui_r)$ let $\ui_1=(j_1,\ldots,j_{d_1})$,
$\ui_2=(h_1,\ldots.h_{d_2})$ etc. and set $\eta^\vee_{0,0}=0$ and
$$
\eta^\vee_{0,1}=\epsilon_{j_1}^\vee,\eta^\vee_{0,2}=\eta^\vee_{0,1}+\epsilon_{j_2}^\vee ,\ldots,
\eta^\vee_{0,d_1-1}=\eta^\vee_{0,d_1-2}+\epsilon_{j_{d_1-1}}^\vee,\eta^\vee_{1,0}=\epsilon_{\ui_1},\eta^\vee_{1,1}=\epsilon_{\ui_1}+
\epsilon_{h_{1}}^\vee,\ldots.
$$
Let
$$
(\eta^\vee_{0,0},\eta^\vee_{0,1},\eta^\vee_{0,2},\ldots,\eta^\vee_{0,d_1-1},\eta^\vee_{1,0},\eta^\vee_{1,1},\ldots,\eta^\vee_{1,d_2-1},
\ldots,\eta^\vee_{r-1,d_r-1}\eta^\vee_{r,0})
$$
be the polyline associated to $w_\gamma$ (viewed as a gallery of type ${\mathbf 1_N}$). The indexing is such that $\eta^\vee_{k,0}=\mu^\vee_{k}$. 
We use now Corollary~\ref{coro1image}
and compare the product of unipotent groups $\bU_{\gamma,k}^-$ and the unipotent group $U_{\mu^\vee_k}^-$ 
associated to the gallery $\gamma$ 
and the product of unipotent groups $\bU^-_{w_\gamma,(k,0)}\cdots \bU^-_{w_\gamma,(k,d_k-1)}$
associated to the gallery $w_\gamma$.

For a positive root $\epsilon_{\ell}-\epsilon_{\ell'}$, $1\le \ell<\ell'\le n$, denote by 
$p^k_{\ell,\ell'}$ the integer such that $\mu_k^\vee\in H_{\epsilon_{\ell}-\epsilon_{\ell'}, -p^k_{\ell,\ell'}}$.
Let $\ui_{k+1}=(j_1,\ldots,j_{d_k})$, then we have for $k=0,\ldots,r-1$:
$$
\bU^-_{\gamma,k}=\prod_{\substack{ j_1< m\le n \\ m\not=j_2,\ldots,j_{d_k}}} U_{-\epsilon_{j_1}+\epsilon_m, p^k_{j_1,m}}
\prod_{\substack{ j_2< m\le n \\ m\not=j_3,\ldots,j_{d_k}}} U_{-\epsilon_{j_2}+\epsilon_m, p^k_{j_2,m}}\cdots 
\prod_{{ j_{d_k}< m\le n}} U_{-\epsilon_{j_{d_k}}+\epsilon_m, p^k_{j_{d_k},m}}.
$$
Note that  $\mu_k^\vee\in H_{\epsilon_{j_t}-\epsilon_m, -p^k_{\ell,m}}$ if and only if $\mu_k^\vee+\epsilon_{j_1}^\vee+\ldots+\epsilon_{j_{t-1}}^\vee\in H_{\epsilon_{j_t}-\epsilon_m, -p^k_{\ell,m}}$,
and hence
$$
\begin{array}{rl}
\bU^-_{w_\gamma,(k,0)}&\cdots \bU^-_{w_\gamma,(k,d_k-1)}\\
=&\prod_{j_1<m\le n }U_{-\epsilon_{j_1}+\epsilon_m, p^k_{j_1,m}}\prod_{j_2<m\le n}U_{-\epsilon_{j_2}+\epsilon_m, p^k_{j_2,m}}\cdots
\prod_{j_{d_k}<m\le n}U_{-\epsilon_{j_{d_k}}+\epsilon_m, p^k_{j_{d_k},m}}.
\end{array}
$$
It follows that $\bU^-_{\gamma,k}\subseteq \bU^-_{w_\gamma,(k,0)}\cdots \bU^-_{w_\gamma,(k,d_k-1)}\subseteq U_{\mu_k}^-$ and hence by Corollary~\ref{coro1image}:
$$
\begin{array}{rcl}
Z_\gamma=\bU^-_{\gamma,0} \bU^-_{\gamma,1} \cdots  \bU^-_{\gamma,{r-1}}.\mu_r^\vee
&\subseteq& 
Z_{w_\gamma}=\bU^-_{w_\gamma,(0,0)}\bU^-_{w_\gamma,(0,1)}\cdots \bU^-_{w_\gamma,(r-1,d_{r}-1)}.\mu_r^\vee \\
&\subseteq&
Z_\gamma=U^-_{\mu^\vee_0} U^-_{\mu^\vee_1} \cdots  U^-_{\mu^\vee_{r-1}}.\mu_r^\vee
\end{array}
$$
This shows that $Z_\gamma$ depends only on the word $w_\gamma$, which proves the proposition. 
\qed
\section{Proof of Theorem~\ref{mainthm1}}\label{proffofmainthm}

We have seen that the image $\pi(C_\gamma)$ of a cell depends only on the associated word, to prove part {\it b)} of
Theorem~\ref{mainthm1} it remains to show that the closure of the image depends only on the equivalence 
class of the associated word modulo the Knuth relation. 
We discuss first the simplest case at length. The proof in the general case is very similar, we leave the details to the reader.
As in the section before, we view a word as a gallery of type $(1,1,\ldots,1)$.
\begin{lem}\label{knuthinbuildung0}
If $w_1=yxz$ and $w_2=yzx$, $x\le y<z$, then 
$$
\pi(C_{w_1})=\bU^-_{w_1,0}\bU^-_{w_1,1}\bU^-_{w_1,2}\cdot(\epsilon_x+\epsilon_y+\epsilon_z)^\vee
\text{\ and\ }
\pi(C_{w_2})=\bU^-_{w_2,0}\bU^-_{w_2,1}\bU^-_{w_2,2}\cdot(\epsilon_x+\epsilon_y+\epsilon_z)^\vee
$$ 
have a common dense subset.
\end{lem}
\begin{proof}
Assume first $x<y<z$.
The key associated to $w_1$ is $\T_1=\Skew(0:z,x,y)$, let $\T_1'$ be the key $\Skew(0:x,y|0:z)$ having the same associated word.
Set $\nu^\vee=(\epsilon_x+\epsilon_y+\epsilon_z)^\vee$, then we have:
$$
\begin{array}{rcl}
{\pi(C_{w_1})}&=&
\bU^-_{\gamma_{\T_1}}\cdot\nu^\vee\\
&=&
\bU^-_{\gamma_{\T_1'}}\cdot\nu^\vee\\
&=&
(\prod_{\substack{\ell>y}}U_{(-\epsilon_y+\epsilon_\ell,0)})
(\prod_{\substack{k>x\\ k\not=y,z}}U_{(-\epsilon_x+\epsilon_k,0)})U_{(-\epsilon_x+\epsilon_y,-1)}
(\prod_{m>z }U_{(-\epsilon_z+\epsilon_m,0)})\cdot\nu^\vee\\
&=&(\prod_{\substack{\ell>y\\ \ell\not=z}}U_{(-\epsilon_y+\epsilon_\ell,0)})U_{(-\epsilon_y+\epsilon_z,0)}
(\prod_{\substack{k>x\\ k\not=y,z}}U_{(-\epsilon_x+\epsilon_k,0)})U_{(-\epsilon_x+\epsilon_y,-1)}\\
&&
\hskip 200pt(\prod_{m>z }U_{(-\epsilon_z+\epsilon_m,0)})\cdot\nu^\vee\\
&=&(\prod_{\substack{\ell>y\\ \ell\not=z}}U_{(-\epsilon_y+\epsilon_\ell,0)})
(\prod_{\substack{k>x\\ k\not=y,z}}U_{(-\epsilon_x+\epsilon_k,0)}) U_{(-\epsilon_y+\epsilon_z,0)}
U_{(-\epsilon_x+\epsilon_y,-1)}\\
&&
\hskip 200pt(\prod_{m>z }U_{(-\epsilon_z+\epsilon_m,0)})\cdot\nu^\vee
\end{array}
$$
The first equality is implied by Corollary~\ref{coro1image}, the second follows from
Proposition~\ref{wordgleichimage}, the third equation is just expressing $\bU_{\gamma_{\T_1'}}$
as a product of the unipotent factors corresponding to the vertices of the gallery $\gamma_{\T_1'}$,
the fourth equation is obtained by separating $U_{(-\epsilon_y+\epsilon_z,0)}$ from the remaining 
factors of the product $(\prod_{\substack{\ell>y\\ \ell\not=z}}U_{(-\epsilon_y+\epsilon_\ell,0)})$,
which is possible since all the factors of this product commute. The last equality is obtained by switching 
the commuting factors $U_{(-\epsilon_y+\epsilon_z,0)}$ and $(\prod_{\substack{k>x\\ k\not=y,z}}U_{(-\epsilon_x+\epsilon_k,0)})$.

The terms $\prod_{m>z }U_{(-\epsilon_z+\epsilon_m,0)}$ and $U_{(-\epsilon_x+\epsilon_y,-1)}$ commute, so we have
$$
{\pi(C_{w_1})}=(\prod_{\substack{\ell>y\\ \ell\not=z}}U_{(-\epsilon_y+\epsilon_\ell,0)})
(\prod_{\substack{k>x\\ k\not=y,z}}U_{(-\epsilon_x+\epsilon_k,0)}) U_{(-\epsilon_y+\epsilon_z,0)}
(\prod_{m>z }U_{(-\epsilon_z+\epsilon_m,0)})U_{(-\epsilon_x+\epsilon_y,-1)}\cdot\nu^\vee.
$$
Recall that for $m>z$ and $s,t\in\bC$ (see \cite{Ste} or \cite{T}):
$$
U_{(-\epsilon_y+\epsilon_z,0)}(s)U_{(-\epsilon_z+\epsilon_m,0)}(t)=U_{(-\epsilon_y+\epsilon_m,0)}(st)
U_{(-\epsilon_z+\epsilon_m,0)}(t)U_{(-\epsilon_y+\epsilon_z,0)}(s)
$$
Now $U_{(-\epsilon_y+\epsilon_m,0)}$, $m>z$, commutes with $U_{(-\epsilon_x+\epsilon_k,0)}$, $k\not=y,z$, $U_{(-\epsilon_y+\epsilon_z,0)}$
and $U_{(-\epsilon_z+\epsilon_q,0)}$, $z<q$, so we can join the factor $U_{(-\epsilon_y+\epsilon_m,0)}(st)$ into the first product (where this
unipotent subgroup occurs already with a free parameter) and 
obtain:
$$
{\pi(C_{w_1})}=(\prod_{\substack{\ell>y\\ \ell\not=z}}U_{(-\epsilon_y+\epsilon_\ell,0)})
(\prod_{\substack{k>x\\ k\not=y,z}}U_{(-\epsilon_x+\epsilon_k,0)}) 
(\prod_{m>z }U_{(-\epsilon_z+\epsilon_m,0)})
U_{(-\epsilon_y+\epsilon_z,0)}U_{(-\epsilon_x+\epsilon_y,-1)}\cdot\nu^\vee.
$$
Now for $s,t\in\bC$ we have
$$
U_{(-\epsilon_y+\epsilon_z,0)}(s)U_{(-\epsilon_x+\epsilon_y,-1)}(t)=
U_{(-\epsilon_x+\epsilon_y,-1)}(t) U_{(-\epsilon_x+\epsilon_z,-1)}(st)U_{(-\epsilon_y+\epsilon_z,0)}(s).
$$
Since $U_{(-\epsilon_y+\epsilon_z,0)}(s)\subseteq P_{\nu^\vee}$ for all $s\in\bC$, we can omit this term and see that
the set
$$
(\prod_{\substack{\ell>y\\ \ell\not=z}}U_{(-\epsilon_y+\epsilon_\ell,0)})
(\prod_{\substack{k>x\\ k\not=y,z}}U_{(-\epsilon_x+\epsilon_k,0)}) 
(\prod_{m>z }U_{(-\epsilon_z+\epsilon_m,0)})
U_{(-\epsilon_x+\epsilon_y,-1)}(*) U_{(-\epsilon_x+\epsilon_z,-1)}(**)\cdot\nu^\vee 
$$
for parameters $*,**\not=0$ forms a dense subset in $\pi(C_{w_1})$. After switching the commuting factors
$(\prod_{\substack{k>x\\ k\not=y,z}}U_{(-\epsilon_x+\epsilon_k,0)}) $ and $(\prod_{m>z }U_{(-\epsilon_z+\epsilon_m,0)})$
we see that
$$
(\prod_{\substack{\ell>y\\ \ell\not=z}}U_{(-\epsilon_y+\epsilon_\ell,0)})
(\prod_{m>z }U_{(-\epsilon_z+\epsilon_m,0)})
(\prod_{\substack{k>x\\ k\not=y,z}}U_{(-\epsilon_x+\epsilon_k,0)}) 
U_{(-\epsilon_x+\epsilon_y,-1)}(*) U_{(-\epsilon_x+\epsilon_z,-1)}(**)\cdot\nu^\vee 
$$ 
is also a dense subset in
$$
\begin{array}{rcl}
{\pi(C_{w_2})}&=&
\bU^-_{\gamma_{\T_2}}\cdot\nu^\vee\\
&=&
\bU^-_{\gamma_{\T_2'}}\cdot\nu^\vee\\
&=&
(\prod_{\substack{\ell>y\\ \ell\not=z}}U_{(-\epsilon_y+\epsilon_\ell,0)})
(\prod_{m>z }U_{(-\epsilon_z+\epsilon_m,0)})
(\prod_{\substack{k>x\\ k\not=y,z}}U_{(-\epsilon_x+\epsilon_k,0)})\\
&&\hskip 150pt U_{(-\epsilon_x+\epsilon_y,-1)} U_{(-\epsilon_x+\epsilon_z,-1)}\cdot\nu^\vee,
\end{array}
$$
where $\T_2=\Skew(0:x,z,y)$ and $\T_2'$ is the key $\Skew(0:x,y|1:z)$ having the same associated word,
which finishes the proof in this case.

Assume now $x=y<z$. The key associated to the word $w_1=xxz$ is a $\T_1=\Skew(0:z,x,x)$,
let $\T_1'$ be the key $\Skew(0:x,x|0:z)$. Set $\nu^\vee=2\epsilon^\vee_x+\epsilon^\vee_z$.
We get as image of the associated cell: 
$$
\pi(C_{w_1})=\bU^-_{\gamma_{\T_1}}\cdot\nu^\vee=\bU^-_{\gamma_{\T_1'}}\cdot\nu^\vee=
\prod_{k>x}U_{(-\epsilon_x+\epsilon_k,0)}\prod_{\substack{\ell>x\\\ell\not=z}}U_{(-\epsilon_x+\epsilon_\ell,1)}
\prod_{m>z}U_{(-\epsilon_z+\epsilon_m,0)}\cdot\nu^\vee.
$$
The key associated to the word $w_2=xzx$ is $\T_2=\Skew(0:x,z,x)$ and we get as image of the associated cell
$$
\begin{array}{rcl}
\pi(C_{w_2})
&=&\bU^-_{\T_2}\cdot\nu^\vee\\
&=&\prod_{k>x}U_{(-\epsilon_x+\epsilon_k,0)}
\prod_{m>z}U_{(-\epsilon_z+\epsilon_m,0)}
\prod_{\substack{\ell>x}}U_{(-\epsilon_x+\epsilon_\ell,1)}\cdot\nu^\vee\\
\end{array}
$$
But $U_{(-\epsilon_x+\epsilon_z,1)}\subset P_{\nu^\vee}$ (since $(2\epsilon^\vee_x+\epsilon^\vee_z,\epsilon_x-\epsilon_z)-1=0$)
and the subgroups $U_{(-\epsilon_x+\epsilon_\ell,1)}$ and $U_{(-\epsilon_z+\epsilon_m,0)}$ commute for
$\ell>x,\ell\not=z$ and $m>z$, so
$$
\pi(C_{w_2})=\bU^-_{\T_2}\cdot\nu^\vee=
\prod_{k>x}U_{(-\epsilon_x+\epsilon_k,0)}
\prod_{\substack{\ell>x\\ \ell\not=z}}U_{(-\epsilon_x+\epsilon_\ell,1)}
\prod_{m>z}U_{(-\epsilon_z+\epsilon_m,0)}\cdot\nu^\vee=\pi(C_{w_1}),
$$
which finishes the proof of the lemma.
\end{proof}
\begin{lem}\label{knuthinbuildung00}
If $w_1=xzy$ and $w_2=zxy$, $x < y\le z$, then 
$$
\pi(C_{w_1})=\bU^-_{w_1,0}\bU^-_{w_1,1}\bU^-_{w_1,2}\cdot(\epsilon_x+\epsilon_y+\epsilon_z)^\vee
\text{\ and\ }
\pi(C_{w_2})=\bU^-_{w_2,0}\bU^-_{w_2,1}\bU^-_{w_2,2}\cdot(\epsilon_x+\epsilon_y+\epsilon_z)^\vee
$$ 
have a common dense subset.
\end{lem}

\begin{proof}
We only give a sketch of proof and leave the details to the reader.
Assume first $x<y<z$.
The key associated to $w_1$ is $\T_1=\Skew(0:y,z,x)$, let $\T_1'$ be the key $\Skew(0:y,x|1:z)$ having the same associated word.
Set $\nu^\vee=(\epsilon_x+\epsilon_y+\epsilon_z)^\vee$, then, as in the proof of Lemma \ref{knuthinbuildung0}, we have, on one side:
$$
\begin{array}{rcl}
{\pi(C_{w_1})}&=&
\bU^-_{\gamma_{\T_1}}\cdot\nu^\vee\\
&=&
\bU^-_{\gamma_{\T_1'}}\cdot\nu^\vee\\
&=&
(\prod_{\substack{\ell>x\\ \ell\ne z}}U_{(-\epsilon_x+\epsilon_\ell,0)})
(\prod_{\substack{m>z}}U_{(-\epsilon_z+\epsilon_m,0)})
(\prod_{\substack{k>y\\ k\ne z}}U_{(-\epsilon_y + \epsilon_k,0)})
U_{(-\epsilon_y+\epsilon_z,-1)}\cdot\nu^\vee\\
\end{array}
$$ On the other side, the key associated to $w_2$ is $\T_2=\Skew(0:y,x,z)$. Let $\T_2'$ be the key $\Skew(0:x,z|0:y)$ having the same associated word. Then, we have:
$$
\begin{array}{rcl}
{\pi(C_{w_1})}&=&
\bU^-_{\gamma_{\T_1}}\cdot\nu^\vee\\
&=&
\bU^-_{\gamma_{\T_1'}}\cdot\nu^\vee\\
&=&
(\prod_{\substack{m>z}}U_{(-\epsilon_z+\epsilon_m,0)})
(\prod_{\substack{\ell>x\\ \ell\ne z,y}}U_{(-\epsilon_x+\epsilon_\ell,0)}) U_{(-\epsilon_x+\epsilon_z,-1)}\\
&&\hskip 130pt (\prod_{\substack{k>y\\ k\ne z}}U_{(-\epsilon_y + \epsilon_k,0)})U_{(-\epsilon_y + \epsilon_k, 0)}
\cdot\nu^\vee\\
\end{array}
$$ 
The same kind of computations as in the proof of Lemma \ref{knuthinbuildung0} shows that those two images share a common dense subset.
\medskip
Now assume that $x< y = y$. Again, using an analogous calculation with the keys $\T_1' = \Skew(0:y,x|1:y)$ and 
$\T_2' = \Skew(0:x,y|0:y)$, one can show that $\pi(C_{w_1})$ and $\pi(C_{w_2})$ have a common dense subset.
\end{proof}

Let now $w_1,w_2$ be words such that
$$
w_1=(i_1,\ldots,i_{r},y,x,z,i_{r+4},\ldots,i_N)\quad w_2=(i_1,\ldots,i_{r},y,z,x,i_{r+4},\ldots,i_N),
$$
where $x\le y<z$, and let $C_{w_1},C_{w_2}\subset \Sigma_{\underbrace{(1,\ldots,1)}_N}$ be the associated cells. Denote by
$$
\pi:\Sigma_{(1,\ldots,1)}\rightarrow X_{N\epsilon_1}
$$
the desingularization map. 
\begin{lem}\label{knuthinbuildung1}
$\overline{\pi(C_{w_1})}=\overline{\pi(C_{w_2})}$.
\end{lem}
\proof
Assume first $x<y<z$.
Let $\gamma_1$ be the gallery of shape ${(1,\ldots,1)}$ associated to $w_1$ and with key tableau $\Skew(0:i_N,\ldots,i_1)$.
Let $\gamma_1'$ be 
the gallery corresponding to the key tableau
$$
\Skew(0:i_N,.,.,.,i_{\splus},x,y,i_{\sminus},.,.,.,i_1|5:z).
$$
Let $\varphi(\gamma_1')=(\mu_0^\vee=0,\mu_1^\vee,\ldots,\mu_{N-1}^\vee)$ be the associated polyline.
Fix $p_{x,k}^{r}, p_{y,\ell}^{r},p_{z,m}^{r}\in \bz$ such that $\mu_{r}^\vee\in H_{(\epsilon_x-\epsilon_k,-p_{x,k}^{r})}$,
$\mu_{r}^\vee\in H_{(\epsilon_y-\epsilon_\ell,-p_{y,\ell}^{r})}$ and $\mu_{r}^\vee\in H_{(\epsilon_z-\epsilon_m,-p_{z,m}^{r})}$.
Recall $\mu_{r+1}^\vee=\mu_{r}^\vee+\epsilon_y^\vee$, so $\mu_{r+1}^\vee\in H_{(\epsilon_x-\epsilon_k,-p_{x,k}^{r})}$ for $k>x$, $k\not=y$,
$\mu_{r+1}^\vee\in H_{(\epsilon_z-\epsilon_m,-p_{z,m}^{r})}$ for $m>z$ and
$\mu_{r+1}^\vee\in H_{(\epsilon_x-\epsilon_y,-p_{x,y}^{r}+1)}$ . 
Using Corollary~\ref{coro1image} and Proposition~\ref{proposition3} we get: 
$$
\begin{array}{rcl}
{\pi(C_{w_1})}&=&\bU^-_{\gamma_1}\cdot\nu^\vee\\
&=&
\bU^-_{\gamma_1'}\cdot\nu^\vee\\
&=&
U^-_{\mu^\vee_0}\cdots U^-_{\mu^\vee_{r-1}}(\prod_{\substack{\ell>y}}U_{(-\epsilon_y+\epsilon_\ell,p_{y,\ell}^{r})})
(\prod_{\substack{k>x\\ k\not=y,z}}U_{(-\epsilon_x+\epsilon_k,p_{x,k}^{r})})
U_{(-\epsilon_x+\epsilon_y,p_{x,y}^{r}-1)} \\
&&\hskip 150pt
(\prod_{m>z }U_{(-\epsilon_z+\epsilon_m,p_{z,m}^{r})})U^-_{\mu^\vee_{r+2}}\cdots U^-_{\mu^\vee_{N-2}}\nu^\vee\\
&=&U^-_{\mu^\vee_0}\cdots U^-_{\mu^\vee_{r-1}}
(\prod_{\substack{\ell>y\\ \ell\not=z}}U_{(-\epsilon_y+\epsilon_\ell,p_{y,\ell}^{r})})
(\prod_{\substack{k>x\\ k\not=y,z}}U_{(-\epsilon_x+\epsilon_k,p_{x,k}^{r})}) U_{(-\epsilon_y+\epsilon_z,p_{y,z}^{r})}\\
&&\hskip 70pt
(\prod_{m>z }U_{(-\epsilon_z+\epsilon_m,-p_{z,m}^{r})})U_{(-\epsilon_x+\epsilon_y,p_{x,y}^{r}-1)}
U^-_{\mu^\vee_{r+2}}\cdots U^-_{\mu^\vee_{N-2}}\nu^\vee,
\end{array}
$$
here we use the same arguments (switching commuting subgroups) as in the proof of Lemma~\ref{knuthinbuildung0}.
Recall that for $m>z$ and $s,t\in\bC$ (see \cite{Ste} or \cite{T}):
$$
\begin{array}{rl}
U_{(-\epsilon_y+\epsilon_z,p_{y,z}^{r})}(s)&U_{(-\epsilon_z+\epsilon_m,p_{z,m}^{r})}(t)\\
&=U_{(-\epsilon_y+\epsilon_m,p_{y,m}^{r})}(st)
U_{(-\epsilon_z+\epsilon_m,p_{z,m}^{r})}(t)U_{(-\epsilon_y+\epsilon_z,p_{y,z}^{r})}(s).
\end{array}
$$
As in the proof of Lemma~\ref{knuthinbuildung0}, after switching and gathering the root subgroups we get:
$$
\begin{array}{rcl}
{\pi(C_{w_1})}&=&U^-_{\mu^\vee_0}\cdots U^-_{\mu^\vee_{r-1}}
(\prod_{\substack{\ell>y\\ \ell\not=z}}U_{(-\epsilon_y+\epsilon_\ell,p_{y,\ell}^{r})})
(\prod_{\substack{k>x\\ k\not=y,z}}U_{(-\epsilon_x+\epsilon_k,p_{x,k}^{r})}) \\ 
&&\hskip 10pt (\prod_{m>z }U_{(-\epsilon_z+\epsilon_m,p_{z,m}^{r})})
U_{(-\epsilon_y+\epsilon_z,p_{y,z}^{r})}U_{(-\epsilon_x+\epsilon_y,p_{x,y}^{r}-1)}
U^-_{\mu^\vee_{r+2}}\cdots U^-_{\mu^\vee_{N-2}}\nu^\vee.
\end{array}
$$
Now for $s,t\in\bC$ we have
$$
\begin{array}{rcl}
U_{(-\epsilon_y+\epsilon_z,p_{y,z}^{r})}(s)U_{(-\epsilon_x+\epsilon_y,p_{x,y}^{r}-1)}(t)&=&
U_{(-\epsilon_x+\epsilon_y,p_{x,y}^{r}-1)}(t) U_{(-\epsilon_x+\epsilon_z,p_{x,z}^{r}-1)}(st)\\
&&\hskip 120pt U_{(-\epsilon_y+\epsilon_z,p_{y,z}^{r})}(s).
\end{array}
$$
Since $U_{(-\epsilon_y+\epsilon_z,p_{y,z}^{r})}(s)\subseteq U^-_{\mu_{r+2}^\vee}$ for all $s\in\bC$ 
(recall that $\mu_{r+2}^\vee=\mu_{r}^\vee+\epsilon_x^\vee+\epsilon_y^\vee+\epsilon_z^\vee)$, we can omit this term and we see 
(after switching commuting terms as in the proof of Lemma~\ref{knuthinbuildung0}) that
the set (with parameters $*,**\not=0$)
\begin{equation}\label{opensetone}
{\mathfrak S}=\begin{array}{r}
U^-_{\mu^\vee_0}\cdots U^-_{\mu^\vee_{r-1}}(\prod_{\substack{\ell>y\\ \ell\not=z}}U_{(-\epsilon_y+\epsilon_\ell,p_{y,\ell}^{r})})
(\prod_{m>z }U_{(-\epsilon_z+\epsilon_m,p_{z,m}^{r})})   
(\prod_{\substack{k>x\\ k\not=y,z}}U_{(-\epsilon_x+\epsilon_k,p_{x,k}^{r})}) \\
U_{(-\epsilon_x+\epsilon_y,p_{x,y}^{r}-1)}(*) U_{(-\epsilon_x+\epsilon_z,p_{x,z}^{r}-1)}(**)
U^-_{\mu^\vee_{r+2}}\cdots U^-_{\mu^\vee_{N-2}}\nu^\vee 
\end{array}
\end{equation}
is a dense subset in ${\pi(C_{w_1})}$. Now let $\gamma_2$ be the gallery of shape ${(1,\ldots,1)}$ associated to $w_2$ 
and let $\gamma_2'$ be the gallery corresponding to the key tableau
$$
\Skew(0:i_N,.,.,.,i_{\splus},x,y,i_{\sminus},.,.,.,i_1|6:z).
$$
We get 
$$
\begin{array}{rcl}
{\pi(C_{w_2})}&=&
\bU^-_{\gamma_{2}}\cdot\nu^\vee\\
&=&
\bU^-_{\gamma'_{2}}\cdot\nu^\vee\\
&=&
U^-_{\mu^\vee_0}\cdots U^-_{\mu^\vee_{r-1}}(\prod_{\substack{\ell>y\\ \ell\not=z}}U_{(-\epsilon_y+\epsilon_\ell,p_{y,\ell}^{r})})
(\prod_{m>z }U_{(-\epsilon_z+\epsilon_m,p_{z,m}^{r})})
\\
&&\hskip 10pt (\prod_{\substack{k>x\\ k\not=y,z}}U_{(-\epsilon_x+\epsilon_k,p_{x,k}^{r})})U_{(-\epsilon_x+\epsilon_y,p_{x,y}^{r}-1)} U_{(-\epsilon_x+\epsilon_z,p_{x,z}^{r}-1)}
U^-_{\mu^\vee_{r+2}}\cdots U^-_{\mu^\vee_{N-2}}\nu^\vee.
\end{array}
$$
It follows that the set ${\mathfrak S}$ in \eqref{opensetone} is a common dense subset of $\pi(C_{w_1})$ and $\pi(C_{w_2})$,
proving the lemma in this case. 

The arguments in the case $x=y<z$ used in Lemma~\ref{knuthinbuildung0} are modified in the same way
to prove that in this case we have $\pi(C_{w_1})=\pi(C_{w_2})$.
\endproof

It remains to consider the second Knuth relation.
Let now $w_1,w_2$ be words such that
$$
w_1=(i_1,\ldots,i_{s-1},x,z,y,i_{s+3},\ldots,i_N)\quad w_2=(i_1,\ldots,i_{s-1},z,x,y,i_{s+3},\ldots,i_N),
$$
where $x< y\le z$, and let $C_{w_1},C_{w_2}\subset \Sigma_{(1,\ldots,1)}$ be the associated cells. 
The proof of the corresponding version of Lemma~\ref{knuthinbuildung1} for the Knuth relation $w_1\sim_K w_2$ is nearly identical to the proof
of the lemma above and is left to the reader. We just state the corresponding version of Lemma~\ref{knuthinbuildung1}
\begin{lem}\label{knuthinbuildung2}
$\overline{\pi(C_{w_1})}=\overline{\pi(C_{w_2})}$.
\end{lem}
\vskip 10pt\noindent
{\it Proof of Theorem~\ref{mainthm1}.} 
The direction {\it i)}$\Rightarrow${\it ii)} of the equivalence in Theorem~\ref{mainthm1}{\it b)} follows from
Lemma~\ref{knuthinbuildung1} and Lemma~\ref{knuthinbuildung2}. 
It was proved in \cite{GL2} (in a more general context) that if $\T$ is a semistandard Young tableaux, 
then $\overline{\pi(C_{\gamma_\T})}$ is a MV-cycle in $X_{\nu^\vee}$ of coweight $\mu^\vee$,
where $\mu^\vee$ is the coweight of the gallery $\gamma_\T$, $\ud=(d_1,\ldots,d_r)$ is the shape of $\T$ and 
$\nu^\vee=\sum_{j=1}^s \omega_{d_j}^\vee$. Moreover, the map $\T\rightarrow \overline{\pi(C_{\gamma_\T})}$
induces a bijection between the semistandard Young tableaux of shape $\ud$ and 
coweight $\mu^\vee$ and the MV-cycles in $X_{\nu^\vee}$ of coweight $\mu^\vee$.
Since an arbitrary gallery is Knuth equivalent
to a unique semistandard Young tableau, this proves the direction {\it i)}$\Leftarrow${\it ii)} of the equivalence
and also part {\it a)} of the theorem. 

\bigskip
{\bf Acknowledgment:} St\'ephane Gaussent thanks the project ANR-09-JCJC-0102-01 for some financial support. 
Peter Littelmann and An Hoa Nguyen have been  supported by the priority program SPP 1388 {\it Representation Theory}
of the German Science Foundation
and the Research Training Group 1269 {\it Global structures in geometry and analysis}.

\end{document}